\documentclass[12pt]{amsart}
\usepackage{amsmath,amssymb,latexsym,cancel,rotating}
\usepackage{graphicx,amssymb,mathrsfs,amsmath,color,fancyhdr,amsthm}
\usepackage[all]{xy}

\newtheorem{theorem}{Theorem}[section]
\newtheorem{lemma}[theorem]{Lemma}

\newtheorem{proposition}[theorem]{Proposition}

\newtheorem{remark}[theorem]{Remark}

  \def\leq{\leqslant}  \def\geq{\geqslant}

\def\Hom{\mbox{\rm Hom}}

\def\dim{\mbox{\rm dim}\,}

\begin{document}

\title[Constructions of highest weight modules via functions]
{Constructions of highest weight modules of double Ringel-Hall algebras via functions}
\thanks{This work was supported by National Natural Science Foundation of China (No. 11701028, 11771445)}

\author[M. Zhao]{Minghui Zhao}
\address{School of Science, Beijing Forestry University, Beijing 100083, P. R. China}
\email{zhaomh@bjfu.edu.cn (M. Zhao)}


\date{\today}


\bibliographystyle{abbrv}

\maketitle

\begin{abstract}
In \cite{zheng2008categorification}, Zheng studied the bounded derived
categories of constructible $\bar{\mathbb{Q}}_l$-sheaves on some algebraic stacks consisting of the representations of a enlarged quiver and categorified the integrable highest weight  modules of the corresponding quantum group by using these categories.
In this paper, we shall generalize Zheng's work to highest weight modules of a subalgebra of the double Ringel-Hall algebra associated to a quiver in a functional version.
\end{abstract}


\section{Introduction}
Let $\mathbf{U}$ be a quantum group and $\mathbf{U}^-$ its negative part. In \cite{Lusztig_Canonical_bases_arising_from_quantized_enveloping_algebra,Lusztig_Quivers_perverse_sheaves_and_the_quantized_enveloping_algebras},
Lusztig gave a categorification of $\mathbf{U}^-$ by using the additive category of perverse sheaves on the varieties $E_{\nu,Q}$ consisting of representations of the quiver $Q$ corresponding to $\mathbf{U}$. The set of isomorphism classes of simple perverse sheaves gives the canonical basis of $\mathbf{U}^-$.
Kashiwara gave a combinatorial construction of this basis (called global crystal basis) in \cite{Kashiwara1993Global}.

Khovanov, Lauda (\cite{Khovanov_Lauda_A_diagrammatic_approach_to_categorification_of_quantum_groups_I}) and Rouquier (\cite{Rouquier_2-Kac-Moody_algebras}) introduced the KLR algebras respectively.
The category of finitely generated projective modules of the corresponding KLR algebra gives a categorification of $\mathbf{U}^-$.
Varagnolo, Vasserot (\cite{Varagnolo_Vasserot_Canonical_bases_and_KLR-algebras}) and
Rouquier (\cite{Rouquier_Quiver_Hecke_algebras_and_2-Lie_algebras}) proved that the set of isomorphism classes of indecomposable projective modules of the KLR algebra can categorify the canonical basis.

Let $L(\omega)$ be an integrable highest weight module of $\mathbf{U}$ with highest weight $\omega$. The canonical basis of $\mathbf{U}^-$ induces the canonical basis of $L(\omega)$.
In \cite{zheng2008categorification}, Zheng gave categorifications of $L(\omega)$ and its canonical basis by using the additive category of perverse sheaves on the varieties $E_{\nu,\hat{Q}}$ consisting of representations of the enlarged quiver $\hat{Q}$.
In \cite{Kang2012Categorification}, Kang and Kashiwara studied the category of finitely generated projective modules of the cyclotomic KLR algebra and gave categorifications of $L(\omega)$ and its canonical basis.

Associated to the category of representations of a quiver $Q$ over a finite field of $q$ elements, Ringel (\cite{Ringel_Hall_algebras_and_quantum_groups}) introduced the twisted Ringel-Hall algebra ${H}^{\ast}_q(Q)$ and its composition subalgebra. He also proved that $\mathbf{U}^-|_{v=\sqrt{q}}$ is isomorphic to the composition subalgebra of ${H}^{\ast}_q(Q)$. Green (\cite{green1995hall}) introduced the comultiplication and Xiao (\cite{xiao1997drinfeld}) introduced the antipode on ${H}^{\ast}_q(Q)$. Under these operators, ${H}^{\ast}_q(Q)$ becomes a Hopf algebra. In \cite{xiao1997drinfeld}, Xiao also introduced the double Ringel-Hall algebra ${D}_q(Q)$ of $Q$ and showed that $\mathbf{U}|_{v=\sqrt{q}}$  is a Hopf subalgebra of ${D}_q(Q)$. The representation theory of ${D}_q(Q)$ is studied by Deng and Xiao in \cite{deng2002double}. Double Ringel-Hall algebra is also studied by Sevenhant and Van den Bergh in \cite{Sevenhant_Van_den_Bergh_On_the_double_of_the_Hall_algebra_of_a_quiver}.

Compared with quantum groups, it is important to give categorifications of the Ringel-Hall algebra ${H}^{\ast}_q(Q)$ and its highest weight modules.
In \cite{XXZ_Ringel-Hall_algebras_beyond_their_quantum_groups}, Xiao, Xu and Zhao generalized Lusztig's categorifical constructions of a quantum group and its canonical basis to the generic form of the whole Ringel-Hall algebra.

Denoted by $\hat{D}_q(Q)$ the subalgebra of ${D}_q(Q)$ generated by $u^+_{i}$, $u^-_{\alpha}$ and $K_{\mu}$. In this paper, we shall generalize Zheng's work to highest weight modules of $\hat{D}_q(Q)$ in a functional version.

We shall introduce a space $\mathcal{F}_{\hat{Q}}$ of functions and three maps $\mathcal{E}^{+}_{\hat{Q},i}$, $\mathcal{E}^{-}_{\hat{Q},\alpha}$ and $\mathcal{K}^{\pm1}_{\hat{Q},i}$ on $\mathcal{F}_{\hat{Q}}$. We will prove that the maps $\mathcal{E}^{+}_{\hat{Q},i}$, $\mathcal{E}^{-}_{\hat{Q},\alpha}$ and $\mathcal{K}^{\pm1}_{\hat{Q},i}$ on $\mathcal{F}_{\hat{Q}}$ satisfy the defining relations of $u^+_{i}$, $u^-_{\alpha}$ and $K_{\pm i}$ in $\hat{D}_q(Q)$. Hence
the space $\mathcal{F}_{\hat{Q}}$ becomes a left $\hat{{D}}(Q)$-module
by defining
$$K_{\pm i}.f=\mathcal{K}^{\pm1}_{\hat{Q},i}(f)\,\,,{u^+_{i}}.f=\mathcal{E}^{+}_{\hat{Q},i}(f)\textrm{ and }{u^-_{\alpha}}.f=\mathcal{E}^{-}_{\hat{Q},\alpha}(f).$$
The key of the proof is checking the following relation
$$\mathcal{E}^{-}_{\hat{Q},\alpha}\mathcal{E}^{+}_{\hat{Q},i}-\mathcal{E}^{+}_{\hat{Q},i}\mathcal{E}^{-}_{\hat{Q},\alpha}=
  \frac{|V_i|}{a_{\alpha}}\sum_{\beta\in\mathcal{P}}a_{\beta}\mathcal{E}^{-}_{\hat{Q},\beta}
  (v^{-\langle{\beta,i}\rangle}g_{i\beta}^{\alpha}\mathcal{K}_{\hat{Q},i}-v^{\langle{\beta,i}\rangle}g_{\beta i}^{\alpha}\mathcal{K}^{-1}_{\hat{Q},i})$$
for any $\alpha\in\mathcal{P}$ and $i\in I$.

In Section 2, we shall recall the definitions of double Ringel-Hall algebras and highest weight modules. In Section 3, we shall introduce the space $\mathcal{F}_{\hat{Q}}$ and the maps $\mathcal{K}_{\hat{Q},i}$, $\mathcal{E}^{+}_{\hat{Q},ni}$ and $\mathcal{E}^{-}_{\hat{Q},\alpha}$ on $\mathcal{F}_{\hat{Q}}$. The main theorem is also given in this section. The proof of the main theorem will be given in Section 4.

\section{Double Ringel-Hall algebras}

\subsection{Ringel-Hall algebras}

In this section, we shall recall the definition of Ringel-Hall algebras (\cite{Ringel_Hall_algebras_and_quantum_groups}).

Let $Q=(I,H,s,t)$ be a quiver, where $I$ is the set of vertices, $H$ is the set of arrows and $s,t:H\rightarrow I$ are two maps sending an arrow $h\in H$ to the start $s(h)$ and terminal $t(h)$ respectively.

Let $k=\mathbb{F}_q$ be a finite field with $q$ elements.
A representation of the quiver $Q$ over $k$ is a pair $(\mathbf{V},x)$, where
\begin{enumerate}
  \item[(1)]$\mathbf{V}=\bigoplus_{i\in I}V_i$ is a finite dimensional $I$-graded vector space;
  \item[(2)]$x=(x_h)_{h\in H}\in\bigoplus_{h\in H}\Hom_k(V_{s(h)},V_{t(h)})$.
\end{enumerate}
A morphism from $(\mathbf{V},x)$ to $(\mathbf{V}',x')$ is an $I$-graded linear map $f=(f_i)_{i\in I}:\mathbf{V}\rightarrow\mathbf{V}'$ such that $f_{t(h)}x_h=x'_hf_{s(h)}$ for any $h\in H$. Denote by $\textrm{rep}_{k}Q$ the abelian category of representations of $Q$ over $k$.
For any representation $V=(\mathbf{V},x)\in\textrm{rep}_{k}Q$, the dimension vector
is defined as $\underline{\dim}{V}=\dim V_ii\in\mathbb{N}I$.

The Euler form on $\mathbb{Z}I$ is defined as
$$\langle\nu,\nu'\rangle=\sum_{i\in I}\nu_i\nu'_i-\sum_{h\in H}\nu_{s(h)}\nu'_{t(h)}$$
and the symmetric Euler form is defined as
$$(\nu,\nu')=\langle\nu,\nu'\rangle+\langle\nu',\nu\rangle$$
for any  $\nu=\sum_{i\in I}\nu_ii$ and $\nu'=\sum_{i\in I}\nu'_ii$.
It is well known that
$$\langle\underline{\dim}{V},\underline{\dim}{V'}\rangle=\dim_{k}\textrm{Hom}_{\textrm{rep}_{k}Q}(V,V')
-\dim_{k}\textrm{Ext}_{\textrm{rep}_{k}Q}(V,V'),$$
for two representations $V=(\mathbf{V},x)$ and $V'=(\mathbf{V}',x')$ in $\textrm{rep}_{k}Q$.

Let $\mathcal{P}$ be the set of isomorphism classes of objects in $\textrm{rep}_{k}Q$.
For any $\alpha\in\mathcal{P}$, choose an object $V_{\alpha}\in\textrm{rep}_{k}Q$ such that the isomorphism classes $[V_{\alpha}]$ is $\alpha$.  Denote by $a_{\alpha}$ for the order of the automorphism group of $V_{\alpha}$. For convenience, the dimension vector of $V_{\alpha}$ is still denoted by $\alpha$.

Given three elements $\alpha_1,\alpha_2$ and $\alpha$ in $\mathcal{P}$, denote by $g^{\alpha}_{\alpha_1\alpha_2}$ the number of subrepresentations $W$ of $V_{\alpha}$ such that $W\simeq V_{\alpha_2}$ and $V/W\simeq V_{\alpha_1}$ in $\textrm{rep}_{k}Q$.
Let $v=\sqrt{q}\in \mathbb{C}$.  The Ringel-Hall algebra ${H}_q(Q)$ is the $\mathbb{Q}(v)$-space with basis $\{u_{\alpha}\,\,|\,\,\alpha\in\mathcal{P}\}$ whose multiplication is given by
\begin{displaymath}
u_{\alpha_1}u_{\alpha_2}=\sum_{\alpha\in\mathcal{P}}g^{\alpha}_{\alpha_1,\alpha_2}u_{\alpha}.
\end{displaymath}
Note that ${H}_q(Q)$ is an associative $\mathbb{Q}(v)$-algebra with unit $u_{0}$, where $0$ denotes the isomorphism class of zero representation.

Define a twisted multiplication on ${H}_q(Q)$ by
\begin{displaymath}
u_{\alpha_1}\ast u_{\alpha_2}=v^{\langle\alpha_1,\alpha_2\rangle}\sum_{\alpha\in\mathcal{P}}g^{\alpha}_{\alpha_1\alpha_2}u_{\alpha},
\end{displaymath}
and ${H}^{\ast}_q(Q)=({H}_q(Q),\ast)$ is called the twisted Ringel-Hall algebra.

\subsection{Constructions of Ringel-Hall algebras via functions}\label{subsection_function1}

In this section, we shall recall the constructions of Ringel-Hall algebras via functions introduced by Lusztig (\cite{Lusztig_Canonical_bases_and_Hall_algebras,Lin_Xiao_Zhang_Representations_of_tame_quivers_and_affine_canonical_bases}).

For any $\nu\in\mathbb{N}I$, fix an $I$-graded vector space $\mathbf{V}=\bigoplus_{i\in I}V_i$ of dimension vector $\nu$. Consider the variety $$E_{\nu,Q}=\bigoplus_{h\in H}\Hom_k(V_{s(h)},V_{t(h)}).$$
The group $G_{\nu,I}=\prod_{i\in I}GL(V_i)$ acts on $E_{\nu,Q}$ by $g.x=gxg^{-1}$.
Let $\mathcal{F}_{G_{\nu,I}}(E_{\nu,Q})$ be the space of $G_{\nu,I}$-invariant constructible functions on $E_{\nu,Q}$ and $$\mathcal{F}_{{Q}}=\bigoplus_{\nu\in\mathbb{N}{I}}\mathcal{F}_{G_{\nu,I}}(E_{\nu,Q}).$$

Fix $I$-graded vector spaces $\mathbf{V}'$, $\mathbf{V}''$ and $\mathbf{V}$ of dimension vector $\alpha$, $\beta$ and $\gamma=\alpha+\beta$ respectively. Consider the following correspondence
$$
\xymatrix{E_{\alpha,Q}\times E_{\beta,Q}&E'\ar[r]^{p_2}\ar[l]_-{p_1}&E''\ar[r]^-{p_3}&E_{\gamma,Q}}.
$$
Here
\begin{enumerate}
  \item[(1)]$E''$ is the variety of all pairs $(x,\mathbf{W})$, where $x\in E_{\gamma,{Q}}$ and $\mathbf{W}$ is a $x$-stable $I$-graded vector subspace of $\mathbf{V}$ with dimension vector $\beta$,
  \item[(2)]$E'$ is the variety of all quadruples $(x, \mathbf{W}, \rho', \rho'')$, where $(x, \mathbf{W})\in E''$ and $\rho': \mathbf{V}/\mathbf{W}\cong \mathbf{V}',$ $\rho'': \mathbf{W}\cong \mathbf{V}''$ are linear isomorphisms,
  \item[(3)]$p_2$ and $p_3$ are natural projections,
  \item[(4)]$p_1(x, \mathbf{W}, \rho', \rho'')=(x',x'')$ such that
      $$x'_h\rho'_{s(h)}=\rho'_{t(h)}x_h\,\textrm{ and }\,
      x''_h\rho''_{s(h)}=\rho''_{t(h)}x_h$$
      for any $h\in H.$
\end{enumerate}

The group $G_{\gamma,I}$ acts on $E''$ by $g.(x, \mathbf{W})=(gxg^{-1}, g\mathbf{W})$ for any $g\in G_{\gamma,I}$.
The groups $G_{\alpha,I}\times G_{\beta,I}$ and $G_{\gamma,I}$ act on $E'$ by $(g_1, g_2).(x, \mathbf{W}, \rho', \rho'')=(x, \mathbf{W}, g_1\rho', g_2\rho'')$ and $g.(x, \mathbf{W}, \rho', \rho'')=(gxg^{-1}, g\mathbf{W}, \rho'g^{-1}, \rho''g^{-1})$ for any $(g_1, g_2)\in G_{\alpha,I}\times G_{\beta,I}$ and $g\in G_{\gamma,I}$. The map $p_1$ is $G_{\alpha,I}\times G_{\beta,I}\times G_{\gamma,I}$-equivariant ($G_{\gamma,I}$ acts on $E_{\alpha,Q}\times E_{\beta,Q}$ trivially)  and $p_2$ is a principal $G_{\alpha,I}\times G_{\beta,I}$-bundle.

Consider the following map $$\textrm{Ind}: \mathcal{F}_{G_{\alpha,I}\times G_{\beta,I}}(E_{\alpha,Q}\times E_{\beta,Q}) \rightarrow\mathcal{F}_{G_{\gamma,I}}(E_{\gamma,Q}),$$ which is the composition of the following maps
$$
\xymatrix{\mathcal{F}_{G_{\alpha,I}\times G_{\beta,I}}(E_{\alpha,Q}\times E_{\beta,Q})\ar[r]^-{p^*_1}&\mathcal{F}_{G_{\alpha,I}\times G_{\beta,I}\times G_{\gamma,I}}(E')\ar[r]^-{(p^*_2)^{-1}}&\mathcal{F}_{G_{\gamma,I}}(E'')\ar[r]^-{(p_3)_{!}}&\mathcal{F}_{G_{\gamma,I}}(E_{\gamma,Q})}.
$$
For two functions $f_1\in\mathcal{F}_{G_{\alpha,I}}(E_{\alpha,Q})$ and $f_2\in\mathcal{F}_{G_{\beta,I}}(E_{\beta,{Q}})$,
let $g(x_1,x_2)=f_1(x_1)f_2(x_2)$ for any $(x_1,x_2)\in E_{\alpha,Q}\times E_{\beta,{Q}}$ and set $$f_1\ast f_2=v^{-m_{\alpha,\beta}}\textrm{Ind}(g)$$
where $m_{\alpha,\beta}=\sum_{h\in H}\alpha_{s(h)}\beta_{t(h)}+\sum_{i\in I}\alpha_{i}\beta_{i}$. The algebra $(\mathcal{F}_{Q},\ast)$ is still denoted by $\mathcal{F}_{Q}$.

For any $\alpha\in\mathcal{P}$, let $\mathcal{O}_{\alpha}$ be the $G_{\alpha,I}$-orbit of $V_{\alpha}$ in $E_{\alpha,Q}$ and $\mathbf{1}_{\alpha}$ the constant function on $\mathcal{O}_{\alpha}$ with value $1$.
It is directed that $$\mathbf{1}_{\alpha}\ast \mathbf{1}_{\beta}(x)=v^{-m_{\alpha,\beta}}g_{\alpha\beta}^{\gamma}$$ for any $x\in\mathcal{O}_{\gamma}$.

\begin{theorem}[\cite{Lusztig_Canonical_bases_and_Hall_algebras,Lin_Xiao_Zhang_Representations_of_tame_quivers_and_affine_canonical_bases}]
There is an isomorphism of algebras $$\varphi:{H}^{\ast}_q(Q)\rightarrow\mathcal{F}_{{Q}}$$ defined by  sending $u_{\alpha}$ to $v^{\textrm{dim}V_{\alpha}-\textrm{dim}G_{\alpha}}\mathbf{1}_{\alpha}$.
\end{theorem}

Denote by $$\Phi_{Q,Q'}:\mathcal{F}_{G_{\nu,I}}(E_{\nu,{Q}})\rightarrow\mathcal{F}_{G_{\nu,I}}(E_{\nu,{Q}'})$$
the Fourier transform, where $Q=(I,H)$ and $Q'=(I,H')$ are two quivers with the same underlying graph and different orientations (\cite{Sevenhant_Van_den_Bergh_On_the_double_of_the_Hall_algebra_of_a_quiver}).

On the relation between the Fourier transforms and the multiplication on $\mathcal{F}_{{Q}}$, we have the following proposition.
\begin{proposition}[\cite{Sevenhant_Van_den_Bergh_On_the_double_of_the_Hall_algebra_of_a_quiver}]\label{proposition_fourier_multi}
For two functions $f_1\in\mathcal{F}_{G_{\alpha,I}}(E_{\alpha,Q})$ and $f_2\in\mathcal{F}_{G_{\beta,I}}(E_{\beta,{Q}})$,
$$\Phi_{Q,Q'}(f_1\ast f_2)=\Phi_{Q,Q'}(f_1)\ast\Phi_{Q,Q'}(f_2),$$
where $Q=(I,H)$ and $Q'=(I,H')$ are two quivers with the same underlying graph and different orientations.
\end{proposition}

\subsection{Double Ringel-Hall algebras}

In this section, we shall recall the definition of  double Ringel-Hall algebra of $Q$ introduced by Xiao (\cite{xiao1997drinfeld}).

Let $\tilde{{H}}_q^+(Q)$ be the free $\mathbb{Q}(v)$-module with basis $\{K_\mu u^+_\alpha\,\,|\,\,\mu\in\mathbb{Z}I, \alpha\in\mathcal{P}\}$.
The multiplication is defined as:
\begin{enumerate}
  \item[(1)]$u^+_{\alpha}u^+_{\beta}=v^{\langle\alpha,\beta\rangle}\sum_{\lambda\in\mathcal{P}}g^{\lambda}_{\alpha\beta}u^+_{\lambda}$ for all $\alpha,\beta\in\mathcal{P}$,
  \item[(2)]$K_{\mu}u^+_\alpha=v^{(\mu,\alpha)}u^+_\alpha K_{\mu}$ for all $\alpha\in\mathcal{P}$ and $\mu\in\mathbb{Z}I$,
  \item[(3)]$K_{\nu}K_{\mu}=K_{\nu+\mu}$ for all $\nu,\mu\in \mathbb{Z}I$,
\end{enumerate}
and the unit is $1=u^+_0=K_0$.
The comultiplication $\Delta$ is defined as:
\begin{enumerate}
  \item[(1)]$\Delta(u^+_{\lambda})=\sum_{\alpha,\beta\in\mathcal{P}}v^{\langle\alpha,\beta\rangle}\frac{a_{\alpha}a_{\beta}}{a_{\lambda}}
  g^{\lambda}_{\alpha\beta}u^+_{\alpha}K_{\beta}\otimes u^+_{\beta}$ for all $\lambda\in\mathcal{P}$,
  \item[(2)]$\Delta(K_{\mu})=K_{\mu}\otimes K_{\mu}$ for all $\mu\in \mathbb{Z}I$,
\end{enumerate}
and the counit $\epsilon$ is defined as:
\begin{enumerate}
  \item[(1)]$\epsilon(u^+_{\lambda})=0$ for all $\lambda\neq0\in\mathcal{P}$,
  \item[(2)]$\epsilon(K_{\mu})=1$ for all $\mu\in \mathbb{Z}I$.
\end{enumerate}
The antipode $S$ is defined as:
\begin{enumerate}
  \item[(1)]$S(u^+_{\lambda})=\delta_{\lambda0}+\sum_{m\geq1}(-1)^m\sum v^{2\sum_{i<j}\langle\lambda_i,\lambda_j\rangle}
  \frac{a_{\lambda_1}\dots a_{\lambda_m}}{a_\lambda}g^{\lambda}_{\lambda_1,\dots,\lambda_m}g^{\pi}_{\lambda_1,\dots,\lambda_m}K_{-\lambda}u^+_{\pi}$ for all $\lambda\in\mathcal{P}$, where $\pi\in\mathcal{P}$, $\lambda_i\in\mathcal{P}-\{0\}$ for any $1\leq i\leq m$,
  \item[(2)]$S(K_{\mu})=K_{-\mu}$ for all $\mu\in \mathbb{Z}I$.
\end{enumerate}
With these operators, $\tilde{{H}}_q^+(Q)$ becomes a Hopf algebra.

The Hopf algebra $\tilde{{H}}_q^+(Q)$ is called the extended twisted Ringel-Hall algebra of $Q$.
Let ${H}_{q}^+(Q)$ be the subalgebra of $\tilde{{H}}_q^+(Q)$ with basis $\{u^+_\alpha\,\,|\,\,\alpha\in\mathcal{P}\}$.

Similarly, one can define the extended twisted Ringel-Hall algebra $\tilde{{H}}^-_q(Q)$ with an $\mathbb{Q}(v)$-basis $\{K_\mu u^-_\alpha\,\,|\,\,\mu\in\mathbb{Z}I, \alpha\in\mathcal{P}\}$. Let ${H}_{q}^-(Q)$ be the subalgebra of $\tilde{{H}}_q^-(Q)$ with basis $\{u^-_\alpha\,\,|\,\,\alpha\in\mathcal{P}\}$.

There is a bilinear form $\varphi:\tilde{{H}}_q^+(Q)\times\tilde{{H}}_q^-(Q)\rightarrow \mathbb{Q}(v)$ such that
$$\varphi(K_{\mu}u^+_\alpha,K_{\nu}u^-_\beta)=v^{-(\mu,\nu)-(\alpha,\nu)+(\mu,\beta)}\frac{|V_\alpha|}{a_{\alpha}}\delta_{\alpha\beta},$$
which is a skew-Hopf pairing. 

Denote by  ${\tilde{D}}_q(Q)$ the Drinfeld double of $(\tilde{{H}}_q^+(Q),\tilde{{H}}_q^-(Q),\varphi)$. The ideal of ${\tilde{D}}_q(Q)$
generated by $\{K_{\mu}\otimes1-1\otimes K_{\mu}\,\,|\,\,\mu\in\mathbb{Z}I\}$ is a Hopf ideal. Denote by ${D}_q(Q)$ the quotient of ${\tilde{D}}_q(Q)$ module this Hopf ideal. Note that ${D}_q(Q)$ is a Hopf algebra, which is called the reduced Drinfeld double of Ringel-Hall algebra of $Q$ or double Ringel-Hall algebra of $Q$ for simplicity.

The double Ringel-Hall algebra admits the following triangular decomposition
$${D}_q(Q)={H}_{q}^-(Q)\otimes{T}\otimes{H}_{q}^+(Q),$$
where ${T}$ is the torus subalgebra generated by $\{K_{\mu}\,\,|\,\,\mu\in\mathbb{Z}I\}$.

Note that there are two isomorphisms of algebras ${^+}:{H}_q^{\ast}(Q)\rightarrow{H}_{q}^+(Q)$ mapping $u_{\lambda}$ to $u^+_{\lambda}$ and ${^-}:{H}_q^{\ast}(Q)\rightarrow{H}_{q}^-(Q)$ mapping $u_{\lambda}$ to $u^-_{\lambda}$ for all $\lambda\in\mathcal{P}$.

In ${D}_q(Q)$, Xiao proved the following formulas.

\begin{proposition}[\cite{xiao1997drinfeld}]\label{proposition_xiao-yang}
For any $i\in I$ and $\alpha\in\mathcal{P}$, we have
\begin{enumerate}
  \item[(1)]$u^-_{i}u^+_{\alpha}-u^+_{\alpha}u^-_{i}=
\frac{|V_i|}{a_{\alpha}}\sum_{\beta\in\mathcal{P}}a_{\beta}u^+_{\beta}
(v^{\langle{\beta,i}\rangle}g_{\beta i}^{\alpha}K_{i}-v^{-\langle{\beta,i}\rangle}g_{i\beta}^{\alpha}K_{-i})$,
  \item[(2)]$u^-_{\alpha}u^+_{i}-u^+_{i}u^-_{\alpha}=
\frac{|V_i|}{a_{\alpha}}\sum_{\beta\in\mathcal{P}}a_{\beta}u^-_{\beta}
(v^{-\langle{\beta,i}\rangle}g_{i\beta}^{\alpha}K_{i}-v^{\langle{\beta,i}\rangle}g_{\beta i}^{\alpha}K_{-i})$.
\end{enumerate}
\end{proposition}

Denoted by $\hat{D}_q(Q)$ the subalgebra of ${D}_q(Q)$ generated by $u^+_{i}(i\in I)$, $u^-_{\alpha}(\alpha\in\mathcal{P})$ and $K_{\mu}(\mu\in \mathbb{Z}I)$.

The definition of ${D}_q(Q)$ and Proposition \ref{proposition_xiao-yang} imply the following proposition.

\begin{proposition}\label{proposition_definition-relation}
The algebra $\hat{D}_q(Q)$ is the associative algebra generated by $u^+_{i}(i\in I)$, $u^-_{\alpha}(\alpha\in\mathcal{P})$ and $K_{\mu}(\mu\in \mathbb{Z}I)$ subject to the following relations:
\begin{enumerate}
  \item[(1)]$K_0=u^+_0=u^-_0=\mathbf{1},\,\,K_{\mu}K_{\nu}=K_{\mu+\nu}$ for any $\mu,\nu\in \mathbb{Z}I$,
  \item[(2)]$K_{\mu}u^+_{i}=v^{(\mu,i)}u^+_{i} K_{\mu}$ for any $i\in I$ and $\mu\in \mathbb{Z}I$,
  \item[(3)]$K_{\mu}u^-_{\beta}=v^{-(\beta,\mu)}u^-_{\beta} K_{\mu}$ for any $\beta\in\mathcal{P}$ and $\mu\in \mathbb{Z}I$,
  \item[(4)]$\sum_{m=0}^{1-(i,j)}(-1)^m(u^+_{i})^{(m)} u^+_{j} (u^+_{i})^{(1-(i,j)-m)}=0$ for any $i\neq j$,
  \item[(5)]$u^-_{\alpha}u^-_{\beta}
  =v^{\langle{\alpha,\beta}\rangle}\sum_{\lambda\in\mathcal{P}}g_{\alpha\beta}^{\lambda}u^-_{\lambda}$ for any $\alpha,\beta\in\mathcal{P}$,
  \item[(6)]$u^-_{\alpha}u^+_{i}-u^+_{i}u^-_{\alpha}=
  \frac{|V_i|}{a_{\alpha}}\sum_{\beta\in\mathcal{P}}a_{\beta}u^-_{\beta}
  (v^{-\langle{\beta,i}\rangle}g_{i\beta}^{\alpha}K_{i}-v^{\langle{\beta,i}\rangle}g_{\beta i}^{\alpha}K_{-i})$ for any $i\in I$ and $\alpha\in \mathcal{P}$,
\end{enumerate}
where $(u^+_{i})^{(m)}=(u^+_{i})^{m}/[m]_{v}!$ and $[m]_{v}!=\prod_{h=1}^{m}\frac{v^h-v^{-h}}{v-v^{-1}}$.
\end{proposition}

\subsection{Quantum groups and composition subalgebras}
In this section, we shall recall the definition of quantum groups (\cite{Lusztig_Introduction_to_quantum_groups}).

Let $a_{ij}=(i,j)$ and $A=(a_{ij})_{i,j\in I}$ be the symmetric generalized Cartan matrix associated to the quiver $Q$. Let $v$ be an indeterminate.

The quantum group $\mathbf{U}$ associated to the quiver $Q$ is an associative algebra over $\mathbb{Q}(v)$ with unit element $1$, generated by the elements $E_i$, $F_i(i\in I)$ and $K_{\mu}(\mu\in \mathbb{Z}I)$ subject to the following relations
\begin{enumerate}
  \item[(1)]$K_{0}=1$ and $K_{\mu}K_{\nu}=K_{\mu+\nu}$ for all $\mu,\nu\in \mathbb{Z}I$,
  \item[(2)]$K_{\mu}E_{i}=v^{(\mu,i)}E_iK_{\mu}$ for all $i\in I$, $\mu\in \mathbb{Z}I$,
  \item[(3)]$K_{\mu}F_{i}=v^{-(\mu,i)}F_iK_{\mu}$ for all $i\in I$, $\mu\in \mathbb{Z}I$,
  \item[(4)]$\sum_{k=0}^{1-a_{ij}}(-1)^{k}E_i^{(k)}E_jE_i^{(1-a_{ij}-k)}=0$ for all $i\neq j\in I$,
  \item[(5)]$\sum_{k=0}^{1-a_{ij}}(-1)^{k}F_i^{(k)}F_jF_i^{(1-a_{ij}-k)}=0$ for all $i\neq j\in I$,
  \item[(6)]$E_iF_j-F_jE_i=\delta_{ij}\frac{K_{i}-K_{-i}}{v-v^{-1}}$ for all $i,j\in I$,
\end{enumerate}
where $E_i^{(n)}=E_i^n/[n]_{v}!$ and $F_i^{(n)}=F_i^n/[n]_{v}!$.

The quantum group $\mathbf{U}$ has the following triangular decomposition
\begin{displaymath}
\mathbf{U}\cong {\mathbf{U}^-}\otimes{\mathbf{U}^{0}}\otimes{\mathbf{U}^{+}},
\end{displaymath}
where $\mathbf{U}^-$, $\mathbf{U}^+$ and $\mathbf{U}^{0}$ are the subalgebras $\mathbf{U}$ generated by $F_i$, $E_i$ and $K_{\mu}$
for all $i\in I$ and $\mu\in \mathbb{Z}I$ respectively.

Denoted by ${D}_c(Q)$ the subalgebra of ${D}_q(Q)$ generated by $u^\pm_{i}(i\in I)$  and $K_{\mu}(\mu\in \mathbb{Z}I)$. The algebra ${D}_c(Q)$ is called the composition subalgebra of ${D}_q(Q)$. Note the ${D}_c(Q)$ is also the subalgebra of ${\hat{D}}_q(Q)$ generated by $u^\pm_{i}(i\in I)$  and $K_{\mu}(\mu\in \mathbb{Z}I)$.

On the relation between the composition subalgebra ${D}_c(Q)$ and quantum group $\mathbf{U}$ of the quiver $Q$, we have the following theorem.
\begin{theorem}[\cite{Ringel_Hall_algebras_and_quantum_groups,green1995hall,xiao1997drinfeld}]
There is an isomorphism of algebras $$\Phi_q:{D}_c(Q)\rightarrow\mathbf{U}|_{v=\sqrt{q}}$$ defined by sending $u^+_{i}$ to $E_i$, $u^-_{i}$ to $-vF_i$ and $K_{\mu}$ to $K_{\mu}$.
\end{theorem}

\subsection{Highest weight modules}

Similarly to the representation theory of a double Ringel-Hall algebra introduced in \cite{deng2002double}, we can consider the representation theory of its subalgebra ${\hat{D}}_q(Q)$.

Given $\omega=\sum_{i\in I}\omega_ii\in\mathbb{N}I$, denote by $M(\omega)$ the Verma module of $\hat{D}_q(Q)$, which is the quotient of $\hat{D}_q(Q)$ by the left ideal generated by ${u^+_{i}}$, $K_i-v^{\omega_i}$ for all $i\in I$. Denote by $\eta_{\omega}\in M(\omega)$ the coset of $K_0$ in $\hat{D}_q(Q)$.

Denote by $L(\omega)$ the quotient of $M(\omega)$ by the left ideal generated by ${(u^-_{i})}^{\omega_i+1}$ for all $i\in I$. The coset of $\eta_{\omega}$ is still denoted by $\eta_{\omega}$ in $L(\omega)$.

The modules $M(\omega)$ and $L(\omega)$ are highest weight modules of $\hat{D}_q(Q)$.

\section{Constructions of highest weight modules via functions}

\subsection{The spaces $\mathcal{F}_{\hat{Q}}$}

For a quiver $Q=(I,H,s,t)$, let $\hat{I}=I\cup\{\hat{i}\,\,|\,\,i\in I\}$ and $\hat{H}=H\cup\{h_i\,\,|\,\,i\in I\}$, where $h_i$ is an arrow that connects $i$ and $\hat{i}$.
The quiver $\hat{Q}=(\hat{I},\hat{H},s,t)$ is called an enlarged quiver of $Q$.

For any $\nu\in\mathbb{N}\hat{I}$, fix an $\hat{I}$-graded vector space $\mathbf{V}=\bigoplus_{i\in\hat{I}}V_i$ of dimension vector $\nu$. Consider the variety $$E_{\nu,\hat{Q}}=\bigoplus_{h\in \hat{H}}\Hom_k(V_{s(h)},V_{t(h)}).$$
The group $G_{\nu,I}$ acts on $E_{\nu,\hat{Q}}$ by $g.x=gxg^{-1}$.
Let $\mathcal{F}_{G_{\nu,I}}(E_{\nu,\hat{Q}})$ be the space of $G_{\nu,I}$-invariant constructible functions on $E_{\nu,\hat{Q}}$.

Denote by $$\Phi_{\hat{Q},\hat{Q}'}:\mathcal{F}_{G_{\nu,I}}(E_{\nu,\hat{Q}})\rightarrow\mathcal{F}_{G_{\nu,I}}(E_{\nu,\hat{Q}'})$$
the Fourier transform, where $\hat{Q}=(\hat{I},\hat{H})$ and $\hat{Q}'=(\hat{I},\hat{H}')$ are two enlarged quivers with the same underlying graph and different orientations (\cite{Sevenhant_Van_den_Bergh_On_the_double_of_the_Hall_algebra_of_a_quiver}).

For any $i\in I$, choose an enlarged quiver $^i\hat{Q}$ such that $i$ is a source. Let ${^i{E}}_{\nu,{^i\hat{Q}}}$ be the subvariety of ${E}_{\nu,{^i\hat{Q}}}$ consisting of all $x=(x_h)_{h\in\hat{H}}$ such that $$\bigoplus_{h\in\hat{H},s(h)=i}x_h:V_i\rightarrow\bigoplus_{h\in\hat{H},s(h)=i}V_{t(h)}$$ is injective.
Denote by ${\mathcal{N}}_{\nu,{^i\hat{Q}},i}$ the subspace of $\mathcal{F}_{G_{\nu,I}}(E_{\nu,{^i\hat{Q}}})$ consisting of the functions $f$ such that $\textrm{supp}{f}\cap{^i{E}}_{\nu,{^i\hat{Q}}}=\emptyset$. Denote by $j_{\nu,{^i\hat{Q}}}:{^i{E}}_{\nu,{^i\hat{Q}}}\rightarrow{E}_{\nu,{^i\hat{Q}}}$ the natural embedding.

For a general enlarged quiver $\hat{Q}$, denote by ${{\mathcal{N}}}_{\nu,\hat{Q},i}$ the subspace of $\mathcal{F}_{G_{\nu,I}}(E_{\nu,\hat{Q}})$ consisting of the functions $f$ such that $\Phi_{{^i\hat{Q}},\hat{Q}}(f)\in{\mathcal{N}}_{\nu,{^i\hat{Q}},i}$.
Since $\Phi_{\hat{Q}',\hat{Q}''}\Phi_{\hat{Q},\hat{Q}'}=\Phi_{\hat{Q},\hat{Q}''}$ for enlarged quivers $\hat{Q},\hat{Q}'$ and $\hat{Q}''$ with the same underlying graph,  the subspace ${{\mathcal{N}}}_{\nu,\hat{Q},i}$ is independent of the choice of $^i\hat{Q}$.

Let $\mathcal{N}_{\nu,\hat{Q}}$ be the subspace of $\mathcal{F}_{G_{\nu,I}}(E_{\nu,\hat{Q}})$ generated by
${\mathcal{N}}_{\nu,\hat{Q},i}$ for various $i\in I$ and
$$\mathcal{F}_{\nu,\hat{Q}}=\mathcal{F}_{G_{\nu,I}}(E_{\nu,\hat{Q}})/\mathcal{N}_{\nu,\hat{Q}}.$$ Let $$\mathcal{F}_{\hat{Q}}=\bigoplus_{\nu\in\mathbb{N}\hat{I}}\mathcal{F}_{\nu,\hat{Q}}.$$

Fix two enlarged quivers $\hat{Q}=(\hat{I},\hat{H})$ and $\hat{Q}'=(\hat{I},\hat{H}')$ with the same underlying graph and different orientations. Since $\Phi_{\hat{Q},\hat{Q}'}({{\mathcal{N}}}_{\nu,\hat{Q},i})={{\mathcal{N}}}_{\nu,\hat{Q}',i}$ for any $i\in{I}$, it holds that $\Phi_{\hat{Q},\hat{Q}'}({{\mathcal{N}}}_{\nu,\hat{Q}})={{\mathcal{N}}}_{\nu,\hat{Q}'}$.
Hence the Fourier transform $\Phi_{\hat{Q},\hat{Q}'}:\mathcal{F}_{G_{\nu,I}}(E_{\nu,\hat{Q}})\rightarrow\mathcal{F}_{G_{\nu,I}}(E_{\nu,\hat{Q}'})$ induces
the following map
$$\Phi_{\hat{Q},\hat{Q}'}:\mathcal{F}_{\nu,\hat{Q}}\rightarrow\mathcal{F}_{\nu,\hat{Q}'}.$$

\begin{remark}
In \cite{zheng2008categorification}, Zheng studied the algebraic stack $[E_{\nu,\hat{Q}}/G_{\nu,I}]$ and the bounded derived
category $\mathcal{D}([E_{\nu,\hat{Q}}/G_{\nu,I}])$ of constructible $\bar{\mathbb{Q}}_l$-sheaves on it.
Then, Zheng studied a subcategory $\mathcal{N}_{\nu}$ of the the derived
category $\mathcal{D}([E_{\nu,\hat{Q}}/G_{\nu,I}])$ and the localized category $\mathfrak{D}_{\nu}=\mathcal{D}([E_{\nu,\hat{Q}}/G_{\nu,I}])/\mathcal{N}_{\nu}$. Here we study $\mathcal{F}_{G_{\nu,I}}(E_{\nu,\hat{Q}})$, the subspace $\mathcal{N}_{\nu,\hat{Q}}$ of $\mathcal{F}_{G_{\nu,I}}(E_{\nu,\hat{Q}})$ and the quotient space $\mathcal{F}_{\nu,\hat{Q}}=\mathcal{F}_{G_{\nu,I}}(E_{\nu,\hat{Q}})/\mathcal{N}_{\nu,\hat{Q}}$, which are the functional versions of $\mathcal{D}([E_{\nu,\hat{Q}}/G_{\nu,I}])$, $\mathcal{N}_{\nu}$ and $\mathfrak{D}_{\nu}$ respectively.
\end{remark}

\subsection{The maps $\mathcal{K}_{\hat{Q},i}$}

For any $i\in I$ and $\nu\in\mathbb{N}\hat{I}$, define $$\mathcal{K}_{\hat{Q},i}=v^{\bar{\nu}_i-\nu_i}\mathrm{Id}:\mathcal{F}_{G_{\nu,I}}(E_{\nu,\hat{Q}})\rightarrow\mathcal{F}_{G_{\nu,I}}(E_{\nu,\hat{Q}}),$$
where $\bar{\nu}_i=\sum_{h\in\hat{H},s(h)=i}\nu_{t(h)}+\sum_{h\in\hat{H},t(h)=i}\nu_{s(h)}-\nu_i$.

On the relation between the maps $\mathcal{K}_{\hat{Q},i}$ and the Fourier transforms, Zheng proved the following proposition.
\begin{proposition}[\cite{zheng2008categorification}]\label{proposition_zheng_K_1}
For any two enlarged quivers $\hat{Q}$ and $\hat{Q}'$ with the same underlying graph, it holds that
$$\Phi_{\hat{Q},\hat{Q}'}\mathcal{K}_{\hat{Q},i}=\mathcal{K}_{\hat{Q}',i}\Phi_{\hat{Q},\hat{Q}'}.$$
\end{proposition}

In \cite{zheng2008categorification}, Zheng also proved the following propositions.
\begin{proposition}[\cite{zheng2008categorification}]\label{proposition_zheng_K_2}
For any $i\in I$ and $\nu\in\mathbb{N}\hat{I}$, it holds that
$$\mathcal{K}_{\hat{Q},i}(\mathcal{N}_{\nu,\hat{Q}})\subset\mathcal{N}_{\nu,\hat{Q}}.$$
\end{proposition}

Hence, the maps $\mathcal{K}_{\hat{Q},i}:\mathcal{F}_{G_{\nu,I}}(E_{\nu,\hat{Q}})\rightarrow\mathcal{F}_{G_{\nu,I}}(E_{\nu,\hat{Q}})$ induce the following maps
$$\mathcal{K}_{\hat{Q},i}:\mathcal{F}_{\nu,\hat{Q}}\rightarrow\mathcal{F}_{\nu,\hat{Q}}.$$

When $i$ is a source of $\hat{Q}$,
define
$${^i\mathcal{K}}_{\hat{Q},i}=v^{\bar{\nu}_i-\nu_i}\mathrm{Id}:\mathcal{F}_{G_{\nu,I}}({^iE}_{\nu,\hat{Q}})\rightarrow\mathcal{F}_{G_{\nu,I}}({^iE}_{\nu,\hat{Q}}),$$
for any $\nu\in\mathbb{N}\hat{I}$.

By the definition of ${^i\mathcal{K}}_{\hat{Q},i}$ and $\mathcal{K}_{\hat{Q},i}$, we have the following commutative diagram
\begin{equation}\label{cd_1}
\xymatrix{\mathcal{F}_{G_{\nu,I}}({E}_{\nu,\hat{Q}})\ar[d]^-{j^\ast_{\nu,{\hat{Q}}}}\ar[r]^-{\mathcal{K}_{\hat{Q},i}}&
\mathcal{F}_{G_{\nu,I}}({E}_{\nu,\hat{Q}})\ar[d]^-{j^\ast_{\nu,{\hat{Q}}}}\\
\mathcal{F}_{G_{\nu,I}}({^iE}_{\nu,\hat{Q}})\ar[r]^-{{^i\mathcal{K}}_{\hat{Q},i}}&
\mathcal{F}_{G_{\nu,I}}({^iE}_{\nu,\hat{Q}}).
}
\end{equation}

\begin{remark}
The maps $\mathcal{K}_{\hat{Q},i}$ are the functional versions of the functors $\mathcal{K}_{\Omega,i}$ in \cite{zheng2008categorification}. Proposition \ref{proposition_zheng_K_1} and \ref{proposition_zheng_K_2} are the functional versions of Proposition 3.3 and 3.4 in \cite{zheng2008categorification} and the proofs are  on the same.
\end{remark}

\subsection{The maps $\mathcal{E}^{+}_{\hat{Q},ni}$}

For any $\nu,\nu'\in\mathbb{N}\hat{I}$ such that $\nu'-\nu\in\mathbb{N}I$, fix $\hat{I}$-graded vector spaces $\mathbf{V}=\bigoplus_{i\in\hat{I}}V_i$ and $\mathbf{V}'=\bigoplus_{i\in\hat{I}}V'_i$ of dimension vector $\nu$ and $\nu'$ respectively.

Let $$F_{\nu\nu'}=\{y\in\bigoplus_{j\in{\hat{I}}}\Hom(V_j,V'_j)\,\,|\,\,y_j\textrm{ is injective and } y_{\hat{j}}=\textrm{Id} \textrm{ for any $j\in I$}\}$$
and $G_{\nu\nu',I}=G_{\nu,I}\times G_{\nu',I}$. The group $G_{\nu\nu',I}$ acts on $F_{\nu\nu'}$
by $(g,g').y=g'yg^{-1}$ for any $y\in F_{\nu\nu'}$ and $(g,g')\in G_{\nu\nu',I}$.

Let
$$Z_{\hat{Q}}=\{(x,x',y)\in E_{\nu,\hat{Q}}\times E_{\nu',\hat{Q}}\times F_{\nu\nu'}\,\,|\,\,x'_hy_{s(h)}=y_{t(h)}x_h\textrm{ for any $h\in\hat{H}$}\}.$$
Consider the following two maps
$p:Z_{\hat{Q}}\rightarrow E_{\nu,\hat{Q}}$
defined as $p(x,x',y)=x$ and
$p':Z_{\hat{Q}}\rightarrow E_{\nu',\hat{Q}}$
defined as $p'(x,x',y)=x'$.

When $i$ is a source of $\hat{Q}$,
let $${^i{Z}}_{\hat{Q}}=\{(x,x',y)\in Z_{\hat{Q}}\,\,|\,\, x'\in{^i{E}}_{\nu',\hat{Q}}\}.$$
Denote by ${^ip}$ and ${^ip'}$ be the restrictions of $p$ and $p'$ to ${^iZ}_{\hat{Q}}$ respectively.

For any $\nu\in\mathbb{N}\hat{I}$ and $i\in I$, let $\nu'=\nu+ni$ and define
$$\mathcal{E}^{+}_{\hat{Q},ni}=v^{-n\nu_i}|G_{\nu'}|^{-1}({^i{p}})_!({^i{p}')^\ast}:\mathcal{F}_{G_{\nu',I}}(E_{\nu',\hat{Q}})\rightarrow\mathcal{F}_{G_{\nu,I}}(E_{\nu,\hat{Q}}).$$

On the relation between the maps $\mathcal{E}^{+}_{\hat{Q},ni}$ and the Fourier transforms, Zheng proved the following proposition.

\begin{proposition}[\cite{zheng2008categorification}]\label{proposition_fourier_e}
For any two enlarged quivers $\hat{Q}$ and $\hat{Q}'$ both having $i$ as sources, it holds that
$$\Phi_{\hat{Q},\hat{Q}'}\mathcal{E}^{+}_{\hat{Q},ni}=\mathcal{E}^{+}_{\hat{Q}',ni}\Phi_{\hat{Q},\hat{Q}'}.
$$
\end{proposition}

In \cite{zheng2008categorification}, Zheng also proved the following proposition.

\begin{proposition}[\cite{zheng2008categorification}]\label{proposition_zheng_e_1}
For any enlarged quiver $\hat{Q}$ such that $i$ is a source, it holds that
$$\mathcal{E}^{+}_{\hat{Q},ni}(\mathcal{N}_{\nu+ni,\hat{Q}})\subset\mathcal{N}_{\nu,\hat{Q}}.$$
\end{proposition}

Hence, the maps $\mathcal{E}^{+}_{\hat{Q},ni}:\mathcal{F}_{G_{\nu+ni,I}}(E_{\nu+ni,\hat{Q}})\rightarrow\mathcal{F}_{G_{\nu,I}}(E_{\nu,\hat{Q}})$ induce the following maps $$\mathcal{E}^{+}_{\hat{Q},ni}:\mathcal{F}_{\nu+ni,\hat{Q}}\rightarrow\mathcal{F}_{\nu,\hat{Q}}.$$

For a general enlarged quiver $\hat{Q}$, choose a new enlarged quiver ${^i\hat{Q}}$ with the same underlying graph and having $i$ as source. Consider the following maps
$$\mathcal{E}^{+}_{\hat{Q},ni}=\Phi_{{^i\hat{Q}},\hat{Q}}\mathcal{E}^{+}_{{^i\hat{Q}},ni}\Phi_{\hat{Q},{^i\hat{Q}}}:\mathcal{F}_{\nu+ni,\hat{Q}}\rightarrow\mathcal{F}_{\nu,\hat{Q}}.$$
Proposition \ref{proposition_fourier_e} implies that the definition of $\mathcal{E}^{+}_{\hat{Q},ni}$ is independent of the choice of ${^i\hat{Q}}$ for any $i\in I$.

When $i$ is a source of $\hat{Q}$, Im${^ip}\subset{^iE}_{\nu,\hat{Q}}$ and Im${^ip'}\subset{^iE}_{\nu',\hat{Q}}$. Hence, we can
define
$${^i\mathcal{E}}^{+}_{\hat{Q},ni}=v^{-n\nu_i}|G_{\nu'}|^{-1}({^i{p}})_!({^i{p}')^\ast}:\mathcal{F}_{G_{\nu',I}}({^iE}_{\nu',\hat{Q}})\rightarrow\mathcal{F}_{G_{\nu,I}}({^iE}_{\nu,\hat{Q}}),$$
for any $\nu\in\mathbb{N}\hat{I}$ and $\nu'=\nu+ni$.

On the relation between ${^i\mathcal{E}}^{+}_{\hat{Q},ni}$ and ${\mathcal{E}}^{+}_{\hat{Q},ni}$, we have the following commutative diagram (\cite{zheng2008categorification})
\begin{equation}\label{cd_2}
\xymatrix{\mathcal{F}_{G_{\nu',I}}({E}_{\nu',\hat{Q}})\ar[d]^-{j^\ast_{\nu',{\hat{Q}}}}\ar[r]^-{{\mathcal{E}}^{+}_{\hat{Q},ni}}&
\mathcal{F}_{G_{\nu,I}}({E}_{\nu,\hat{Q}})\\
\mathcal{F}_{G_{\nu',I}}({^iE}_{\nu',\hat{Q}})\ar[r]^-{{^i\mathcal{E}}^{+}_{\hat{Q},ni}}&
\mathcal{F}_{G_{\nu,I}}({^iE}_{\nu,\hat{Q}})\ar[u]_-{(j_{\nu,{\hat{Q}}})_{!}}.
}
\end{equation}

\begin{remark}
The maps ${\mathcal{E}}^{+}_{\hat{Q},ni}$ are the functional versions of the functors $\mathcal{E}^{(n)}_{\Omega,i}$ in \cite{zheng2008categorification}. Proposition \ref{proposition_fourier_e} and \ref{proposition_zheng_e_1} are the functional versions of Proposition 3.3 and 3.4 in \cite{zheng2008categorification} and the proofs are on the same.
\end{remark}

\subsection{The maps $\mathcal{E}^{-}_{\hat{Q},\alpha}$}

For any $\alpha\in\mathcal{P}$ and $\nu\in\mathbb{N}\hat{I}$, let $\nu'=\nu+\alpha$.
Fix $\hat{I}$-graded vector spaces $\mathbf{V}=\bigoplus_{i\in\hat{I}}V_i$ and $\mathbf{V}'=\bigoplus_{i\in\hat{I}}V'_i$ of dimension vector $\nu$ and $\nu'$ respectively.

Let
$$Z_{\hat{Q},\alpha}=\{(x,x',y)\in Z_{\hat{Q}}\,\,|\,\,[V_{x'}/\textrm{Im}y]=\alpha\},$$
where $V_{x'}=(\mathbf{V}',x')$.
The restrictions of $p$ and $p'$ are still denoted by $p$ and $p'$.
Define
$$\mathcal{E}^{-}_{\hat{Q},\alpha}=
(-v)^{\textrm{dim}V_{\alpha}}v^{\langle{\alpha,\nu}\rangle+\textrm{dim}V_{\alpha}}|G_{\nu}|^{-1}
p'_!p^\ast:\mathcal{F}_{G_{\nu,I}}(E_{\nu,\hat{Q}})\rightarrow\mathcal{F}_{G_{\nu',I}}(E_{\nu',\hat{Q}}).$$

\begin{remark}
When $\alpha=i$, the map $\mathcal{E}^{-}_{\hat{Q},\alpha}$ is the functional version of the functor $\mathcal{F}_{\Omega,i}$ in \cite{zheng2008categorification}.
\end{remark}

Similarly to the definition of the multiplication $\ast$ on $\mathcal{F}_{Q}$ in Section \ref{subsection_function1}, we can give an alternative description of $\mathcal{E}^{-}_{\hat{Q},\alpha}$.

Fix an $\hat{I}$-graded vector space $\mathbf{V}_{\alpha}$ of dimension vector $\alpha$.
Consider the following correspondence
$$
\xymatrix{E_{\alpha,Q}\times E_{\nu,\hat{Q}}&E'\ar[r]^{p_2}\ar[l]_-{p_1}&E''\ar[r]^-{p_3}&E_{\nu',\hat{Q}}}.
$$
Here
\begin{enumerate}
  \item[(1)]$E''$ is the variety of all pairs $(x,\mathbf{W})$, where $x\in E_{\nu',\hat{Q}}$ and $\mathbf{W}$ is a $x$-stable $\hat{I}$-graded vector subspace of $\mathbf{V}'$ with dimension vector $\nu$;
  \item[(2)]$E'$ is the variety of all quadruples $(x, \mathbf{W}, \rho_1, \rho_2)$ where $(x, \mathbf{W})\in E''$ and $\rho_1: \mathbf{V}/\mathbf{W}\cong \mathbf{V}_\alpha,$ $\rho_2: \mathbf{W}\cong \mathbf{V}$ are linear isomorphisms;
  \item[(3)]$p_2$ and $p_3$ are natural projections;
  \item[(4)] $p_1(x, \mathbf{W}, \rho_1, \rho_2)=(x',x'')$ such that
       $$x'_h(\rho_1)_{s(h)}=(\rho_1)_{t(h)}x_h\,\textrm{ and }\,
       x''_h(\rho_2)_{s(h)}=(\rho_2)_{t(h)}x_h$$
       for any $h\in \hat{H}.$
\end{enumerate}

The group $G_{\nu',I}$ acts on $E''$ by $g.(x, \mathbf{W})=(gxg^{-1}, g\mathbf{W})$ for any $g\in G_{\nu',I}$.
The groups $G_{\alpha,I}\times G_{\nu,I}$ and $G_{\nu',I}$ act on $E'$ by $(g_1, g_2).(x, \mathbf{W}, \rho_1, \rho_2)=(x, \mathbf{W}, g_1\rho_1, g_2\rho_2)$ and $g.(x, \mathbf{W}, \rho_1, \rho_2)=(gxg^{-1}, g\mathbf{W}, \rho_1g^{-1}, \rho_2g^{-1})$ for any $(g_1, g_2)\in G_{\alpha,I}\times G_{\nu,I}$ and $g\in G_{\nu',I}$. The map $p_1$ is $G_{\alpha,I}\times G_{\nu,I}\times G_{\nu',I}$-equivariant ($G_{\nu',I}$ acts on $E_{\alpha,Q}\times E_{\nu,\hat{Q}}$ trivially) and  $p_2$ is a principal $G_{\alpha,I}\times G_{\nu,I}$-bundle.

Consider the following map $$\textrm{Ind}: \mathcal{F}_{G_{\alpha,I}\times G_{\nu,I}}(E_{\alpha,Q}\times E_{\nu,\hat{Q}}) \rightarrow\mathcal{F}_{G_{\nu',I}}(E_{\nu',\hat{Q}})$$ as the composition of the following maps
$$
\xymatrix{\mathcal{F}_{G_{\alpha,I}\times G_{\nu,I}}(E_{\alpha,Q}\times E_{\nu,\hat{Q}})\ar[r]^-{p^*_1}&\mathcal{F}_{G_{\alpha,I}\times G_{\nu,I}\times G_{\nu',I}}({E}')\ar[r]^-{(p_2^\ast)^{-1}}&\mathcal{F}_{G_{\nu',I}}({E}'')\ar[r]^-{(p_3)_{!}}&\mathcal{F}_{G_{\nu',I}}(E_{\nu',\hat{Q}})}.
$$

For two functions $f_1\in\mathcal{F}_{G_{\alpha,I}}(E_{\alpha,Q})$ and $f_2\in\mathcal{F}_{G_{\nu,I}}(E_{\nu,\hat{Q}})$,
let $g(x_1,x_2)=f_1(x_1)f_2(x_2)$ for any $(x_1,x_2)\in E_{\alpha,Q}\times E_{\nu,\hat{Q}}$. Then $g\in\mathcal{F}_{G_{\alpha,I}\times G_{\nu,I}}(E_{\alpha,Q}\times E_{\nu,\hat{Q}})$.
Set $$f_1\ast f_2=v^{-m_{\alpha,\nu}}\textrm{Ind}(g).$$

\begin{proposition}\label{proposition_relation_multi}
For any $\alpha\in\mathcal{P}$ and $f\in\mathcal{F}_{G_{\nu,I}}(E_{\nu,\hat{Q}})$, we have $$\mathcal{E}^{-}_{\hat{Q},\alpha}(f)=(-v)^{\textrm{dim}V_{\alpha}}v^{\langle{\alpha,\nu}\rangle+\textrm{dim}V_{\alpha}+m_{\alpha,\nu}}\mathbf{1}_{\alpha}\ast f.$$
\end{proposition}

\begin{proof}
For any $\beta\in\mathcal{P}$ with dimension vector $\nu$, $\gamma\in\mathcal{P}$ with dimension vector $\nu'=\nu+\alpha$ and $x_{\gamma}\in\mathcal{O}_{\gamma}$,
\begin{eqnarray*}\mathcal{E}^{-}_{\hat{Q},\alpha}(\mathbf{1}_{\beta})(x_{\gamma})&=&
(-v)^{\textrm{dim}V_{\alpha}}v^{\langle{\alpha,\nu}\rangle+\textrm{dim}V_{\alpha}}
|G_{\nu}|^{-1}
p'_!p^\ast(\mathbf{1}_{\beta})(x_{\gamma})\\
&=&
(-v)^{\textrm{dim}V_{\alpha}}v^{\langle{\alpha,\nu}\rangle+\textrm{dim}V_{\alpha}}
|G_{\nu}|^{-1}
\sum_{(x,x_{\gamma},y)\in Z_{\hat{Q},\alpha}}p^\ast(\mathbf{1}_{\beta})(x,x_{\gamma},y)\\
&=&
(-v)^{\textrm{dim}V_{\alpha}}v^{\langle{\alpha,\nu}\rangle+\textrm{dim}V_{\alpha}}
|G_{\nu}|^{-1}
\sum_{(x,x_{\gamma},y)\in Z_{\hat{Q},\alpha}}\mathbf{1}_{\beta}(x)\\
&=&
(-v)^{\textrm{dim}V_{\alpha}}v^{\langle{\alpha,\nu}\rangle+\textrm{dim}V_{\alpha}}
|G_{\nu}|^{-1}
|(x,x_{\gamma},y)\in Z_{\hat{Q},\alpha},x\in\mathcal{O}_{\beta}|\\
&=&
(-v)^{\textrm{dim}V_{\alpha}}v^{\langle{\alpha,\nu}\rangle+\textrm{dim}V_{\alpha}}
|G_{\nu}|^{-1}
|G_{\nu}|g_{\alpha\beta}^{\gamma}\\
&=&(-v)^{\textrm{dim}V_{\alpha}}v^{\langle{\alpha,\nu}\rangle+\textrm{dim}V_{\alpha}+m_{\alpha,\nu}}\mathbf{1}_{\alpha}\ast\mathbf{1}_{\beta}(x_{\gamma}).
\end{eqnarray*}
That is $\mathcal{E}^{-}_{\hat{Q},\alpha}(\mathbf{1}_{\beta})=(-v)^{\textrm{dim}V_{\alpha}}v^{\langle{\alpha,\nu}\rangle+\textrm{dim}V_{\alpha}+m_{\alpha,\nu}}\mathbf{1}_{\alpha}\ast \mathbf{1}_{\beta}$.

For any $f\in\mathcal{F}_{G_{\nu,I}}(E_{\nu,\hat{Q}})$, $f$ is a linear combination of $\mathbf{1}_{\beta}$ for various $\beta\in\mathcal{P}$ and we have  $\mathcal{E}^{-}_{\hat{Q},\alpha}(f)=(-v)^{\textrm{dim}V_{\alpha}}v^{\langle{\alpha,\nu}\rangle+\textrm{dim}V_{\alpha}+m_{\alpha,\nu}}\mathbf{1}_{\alpha}\ast f$.
\end{proof}

Similarly to Proposition \ref{proposition_fourier_multi}, we have the following proposition.
\begin{proposition}[\cite{Sevenhant_Van_den_Bergh_On_the_double_of_the_Hall_algebra_of_a_quiver}]\label{proposition_fourier_multi-1}
For two functions $f_1\in\mathcal{F}_{G_{\alpha,I}}(E_{\alpha,Q})$ and $f_2\in\mathcal{F}_{G_{\nu,I}}(E_{\nu,\hat{Q}})$,
$$\Phi_{\hat{Q},\hat{Q}'}(f_1\ast f_2)=\Phi_{Q,Q'}(f_1)\ast\Phi_{\hat{Q},\hat{Q}'}(f_2),$$
where $Q=(I,H)$ and $Q'=(I,H')$ are two quivers with the same underlying graph and different orientations, $\hat{Q}=(\hat{I},\hat{H})$ and $\hat{Q}'=(\hat{I},\hat{H}')$ are enlarged quivers of $Q$ and $Q'$ respectively.
\end{proposition}

By using Proposition \ref{proposition_relation_multi} and \ref{proposition_fourier_multi-1}, we have the following proposition.

\begin{proposition}\label{proposition_N_f}
For any enlarged quiver $\hat{Q}$, we have
$$\mathcal{E}^{-}_{\hat{Q},\alpha}(\mathcal{N}_{\nu,\hat{Q}})\subset\mathcal{N}_{\nu+\alpha,\hat{Q}}.$$
\end{proposition}

\begin{proof}
Let $a(v)=(-v)^{\textrm{dim}V_{\alpha}}v^{\langle{\alpha,\nu}\rangle+\textrm{dim}V_{\alpha}+m_{\alpha,\nu}}$.

When $i\in{I}$ is a source of $\hat{Q}$, it is clear that $\mathcal{E}^{-}_{\hat{Q},\alpha}(f)=a(v)\mathbf{1}_{\alpha}\ast f\in\mathcal{N}_{\nu+\alpha,\hat{Q},i}$  for any $f\in\mathcal{N}_{\nu,\hat{Q},i}$.

For a general enlarged quiver $\hat{Q}$, choose a new enlarged quiver ${^i\hat{Q}}$ with the same underlying graph and having $i$ as source.
For any $f\in\mathcal{N}_{\nu,\hat{Q},i}$, there is a function $f'\in{{\mathcal{N}}}_{\nu,{^i\hat{Q}},i}$ such that $f=\Phi_{{^i\hat{Q}},\hat{Q}}(f')$.
Hence $$\mathcal{E}^{-}_{\hat{Q},\alpha}(f)=a(v)\mathbf{1}_{\alpha}\ast f=a(v)\mathbf{1}_{\alpha}\ast \Phi_{{^i\hat{Q}},\hat{Q}}(f').$$
Since $\mathbf{1}_{\alpha}=\Phi_{{^i{Q}},{Q}}\Phi_{{{Q}},{^i{Q}}}(\mathbf{1}_{\alpha})$,
$$\mathcal{E}^{-}_{\hat{Q},\alpha}(f)=a(v)\Phi_{{^i{Q}},{Q}}\Phi_{{{Q}},{^i{Q}}}(\mathbf{1}_{\alpha})\ast \Phi_{{^i\hat{Q}},{\hat{Q}}}(f')=a(v)\Phi_{{^i\hat{Q}},{\hat{Q}}}(\Phi_{{{Q}},{^i{Q}}}(\mathbf{1}_{\alpha})\ast f').$$
Note that $\Phi_{{{Q}},{^i{Q}}}(\mathbf{1}_{\alpha})=\sum_{\gamma\in\mathcal{P}}a_{\gamma}\mathbf{1}_{\gamma}$ for some $a_{\gamma}\in\mathbb{C}$.
Hence $$\mathcal{E}^{-}_{\hat{Q},\alpha}(f)=a(v)\Phi_{{^i\hat{Q}},{\hat{Q}}}(\sum_{\gamma\in\mathcal{P}}a_{\gamma}(\mathbf{1}_{\gamma}\ast f')).$$
Since $\mathbf{1}_{\gamma}\ast f'\in\mathcal{N}_{\nu+\alpha,{^i\hat{Q}},i}$, $\mathcal{E}^{-}_{\hat{Q},\alpha}(f)\in\mathcal{N}_{\nu+\alpha,\hat{Q},i}$.

By the definition of $\mathcal{N}_{\nu,\hat{Q}}$, we get the desired result.
\end{proof}

Hence, the maps $\mathcal{E}^{-}_{\hat{Q},\alpha}:\mathcal{F}_{G_{\nu,I}}(E_{\nu,\hat{Q}})\rightarrow\mathcal{F}_{G_{\nu+\alpha,I}}(E_{\nu+\alpha,\hat{Q}})$ induce the following maps
$$\mathcal{E}^{-}_{\hat{Q},\alpha}:\mathcal{F}_{\nu,\hat{Q}}\rightarrow\mathcal{F}_{\nu+\alpha,\hat{Q}}.$$

When $i$ is a source of $\hat{Q}$,
let
$${^iZ}_{\hat{Q},\alpha}=\{(x,x',y)\in{Z}_{\hat{Q},\alpha}\,\,|\,\,x'\in{^i{E}}_{\nu',\hat{Q}}\}.$$
Denote by ${^ip}$ and ${^ip'}$ be the restrictions of $p$ and $p'$ to ${^iZ}_{\hat{Q},\alpha}$ respectively.

For any $\alpha\in\mathcal{P}$ such that $V_i$ is not a direct summand of $V_{\alpha}$, define
$${^i\mathcal{E}}^{-}_{\hat{Q},\alpha}=(-v)^{\textrm{dim}V_{\alpha}}v^{\langle{\alpha,\nu}\rangle+\textrm{dim}V_{\alpha}}|G_{\nu}|^{-1}({^ip'})_!({^ip})^\ast:
\mathcal{F}_{G_{\nu,I}}({^iE}_{\nu,\hat{Q}})\rightarrow\mathcal{F}_{G_{\nu',I}}({^iE}_{\nu',\hat{Q}}).$$

By Proposition \ref{proposition_N_f}, we have the following commutative diagram on the relation between ${^i\mathcal{E}}^{-}_{\hat{Q},\alpha}$ and ${\mathcal{E}}^{-}_{\hat{Q},\alpha}$
\begin{equation}\label{cd_3}
\xymatrix{\mathcal{F}_{G_{\nu,I}}({E}_{\nu,\hat{Q}})\ar[r]^-{{\mathcal{E}}^{-}_{\hat{Q},\alpha}}\ar[d]^-{j^\ast_{\nu,{\hat{Q}}}}&
\mathcal{F}_{G_{\nu',I}}({E}_{\nu',\hat{Q}})\ar[d]^-{j^\ast_{\nu',{\hat{Q}}}}\\
\mathcal{F}_{G_{\nu,I}}({^iE}_{\nu,\hat{Q}})\ar[r]^-{{^i\mathcal{E}}^{-}_{\hat{Q},\alpha}}&
\mathcal{F}_{G_{\nu',I}}({^iE}_{\nu',\hat{Q}})
.}
\end{equation}

\subsection{Main results}

\begin{theorem}\label{MT0}
On the maps $\mathcal{E}^{+}_{\hat{Q},i}(i\in I)$, $\mathcal{E}^{-}_{\hat{Q},\alpha}(\alpha\in\mathcal{P})$ and $\mathcal{K}^{\pm1}_{\hat{Q},i}(i\in I)$, we have the following relations:
\begin{enumerate}
  \item[(1)]$\mathcal{K}_{\hat{Q},i}\mathcal{K}_{\hat{Q},j}=\mathcal{K}_{\hat{Q},j}\mathcal{K}_{\hat{Q},i}$ for any $i,j\in I$,
  \item[(2)]$\mathcal{K}_{\hat{Q},j}\mathcal{E}^{+}_{\hat{Q},i}=v^{(j,i)}\mathcal{E}^{+}_{\hat{Q},i}\mathcal{K}_{\hat{Q},j}$ for any $i,j\in I$,
  \item[(3)]$\mathcal{K}_{\hat{Q},j}\mathcal{E}^{-}_{\hat{Q},\alpha}=v^{-(\alpha,j)}\mathcal{E}^{-}_{\hat{Q},\alpha}\mathcal{K}_{\hat{Q},j}$ for any $j\in I$ and $\alpha\in\mathcal{P}$,
  \item[(4)]$\sum_{m=0}^{1-(i,j)}(-1)^m(\mathcal{E}^{+}_{\hat{Q},i})^{(m)}\mathcal{E}^{+}_{\hat{Q},j}(\mathcal{E}^{+}_{\hat{Q},i})^{(1-(i,j)-m)}=0$ for any $i\neq j\in I$,
  \item[(5)]$\mathcal{E}^{-}_{\hat{Q},\alpha}\mathcal{E}^{-}_{\hat{Q},\beta}=v^{\langle{\alpha,\beta}\rangle}\sum_{\lambda\in\mathcal{P}}g_{\alpha\beta}^{\lambda}\mathcal{E}^{-}_{\hat{Q},\lambda}$
      for any $\alpha,\beta\in\mathcal{P}$,
  \item[(6)]$\mathcal{E}^{-}_{\hat{Q},\alpha}\mathcal{E}^{+}_{\hat{Q},i}-\mathcal{E}^{+}_{\hat{Q},i}\mathcal{E}^{-}_{\hat{Q},\alpha}=
  \frac{|V_i|}{a_{\alpha}}\sum_{\beta\in\mathcal{P}}a_{\beta}\mathcal{E}^{-}_{\hat{Q},\beta}
  (v^{-\langle{\beta,i}\rangle}g_{i\beta}^{\alpha}\mathcal{K}_{\hat{Q},i}-v^{\langle{\beta,i}\rangle}g_{\beta i}^{\alpha}\mathcal{K}^{-1}_{\hat{Q},i})$ for any $\alpha\in\mathcal{P}$ and $i\in I$.
\end{enumerate}
\end{theorem}

The Relation (1)-(3) are directed from the definitions. The Relation (4) is Theorem 3.9 (8) in \cite{zheng2008categorification}.
The proofs of Relation (5) and (6) will be given in Section \ref{proof_mt1&2}.

As a corollary of Theorem \ref{MT0} and Proposition \ref{proposition_definition-relation}, we have the following theorem.

\begin{theorem}\label{MT1}
By defining
$$K_{\pm i}.f=\mathcal{K}^{\pm1}_{\hat{Q},i}(f)\,\,,{u^+_{i}}.f=\mathcal{E}^{+}_{\hat{Q},i}(f)\textrm{ and }{u^-_{\alpha}}.f=\mathcal{E}^{-}_{\hat{Q},\alpha}(f),$$
the space $\mathcal{F}_{\hat{Q}}$ becomes a left $\hat{{D}}(Q)$-module.
\end{theorem}

For any $\omega=\sum_{i\in I}\omega_i{i}\in\mathbb{Z}I$, let $\mathbf{1}_{\omega}$ be the constant function on ${E}_{\hat{\omega},\hat{Q}}$ with value $1$, where $\hat{\omega}=\sum_{i\in I}\omega_i\hat{i}$.
Denote by $\mathcal{F}_{\hat{Q},\omega}$ the submodule of $\mathcal{F}_{\hat{Q}}$ generated by $\mathbf{1}_{\omega}$.
By the definition, $\mathcal{F}_{\hat{Q},\omega}$ is a highest weight module of $\hat{D}_q(Q)$ with highest weigh $\omega$.

For the relation between $\mathcal{F}_{\hat{Q},\omega}$ and $L(\omega)$, we have the following theorem.

\begin{theorem}\label{MT2}
There is an epimorphism of $\hat{{D}}(Q)$-modules from $L(\omega)$ to $\mathcal{F}_{\hat{Q},\omega}$.
\end{theorem}
\begin{proof}
There is an epimorphism of $\hat{{D}}(Q)$-modules from $\hat{{D}}(Q)$ to $\mathcal{F}_{\hat{Q},\omega}$. Since the image of the left ideal generated by ${u^+_{i}}$, $K_i-v^{\omega_i}$ and $(u^-_{i})^{\omega_i+1}$ for all $i\in I$ is $0$,  there exists an epimorphism of $\hat{{D}}(Q)$-modules from $L(\omega)$ to $\mathcal{F}_{\hat{Q},\omega}$.
\end{proof}

\section{The proofs of Relation (5) and (6) in Theorem \ref{MT0}}\label{proof_mt1&2}

\begin{proposition}[Relation (5) in Theorem \ref{MT0}]
For any $\alpha,\beta\in\mathcal{P}$, we have $$\mathcal{E}^{-}_{\hat{Q},\alpha}\mathcal{E}^{-}_{\hat{Q},\beta}=v^{\langle{\alpha,\beta}\rangle}\sum_{\gamma\in\mathcal{P}}g_{\alpha\beta}^{\gamma}\mathcal{E}^{-}_{\hat{Q},\gamma}.$$
\end{proposition}

\begin{proof}
By Proposition \ref{proposition_relation_multi}, we have $\mathcal{E}^{-}_{\hat{Q},\alpha}(f)=(-v)^{\textrm{dim}V_{\alpha}}v^{\langle{\alpha,\nu}\rangle+\textrm{dim}V_{\alpha}+m_{\alpha,\nu}}\mathbf{1}_{\alpha}\ast f$ for any $f\in\mathcal{F}_{G_{\nu,I}}(E_{\nu,\hat{Q}})$.

Similarly to the associativity of multiplication of a Hall algebra, we have
$$\mathbf{1}_{\alpha}\ast(\mathbf{1}_{\beta}\ast f)=(\mathbf{1}_{\alpha}\ast \mathbf{1}_{\beta})\ast f.$$
Hence,
\begin{eqnarray*}
&&\mathcal{E}^{-}_{\hat{Q},\alpha}\mathcal{E}^{-}_{\hat{Q},\beta}(f)\\
&=&(-v)^{\textrm{dim}V_{\alpha}}v^{\langle{\alpha,\nu+\beta}\rangle+\textrm{dim}V_{\alpha}+m_{\alpha,\nu+\beta}}
(-v)^{\textrm{dim}V_{\beta}}v^{\langle{\beta,\nu}\rangle+\textrm{dim}V_{\beta}+m_{\beta,\nu}}
\mathbf{1}_{\alpha}\ast(\mathbf{1}_{\beta}\ast f)\\
&=&(-v)^{\textrm{dim}V_{\alpha+\beta}}v^{\langle{\alpha+\beta,\nu}\rangle+\langle{\alpha,\beta}\rangle+\textrm{dim}V_{\alpha+\beta}+m_{\alpha+\beta,\nu}+m_{\alpha,\beta}}
(\mathbf{1}_{\alpha}\ast\mathbf{1}_{\beta})\ast f\\
&=&\sum_{\gamma\in\mathcal{P}}(-v)^{\textrm{dim}V_{\gamma}}v^{\langle{\gamma,\nu}\rangle+\langle{\alpha,\beta}\rangle+\textrm{dim}V_{\gamma}+m_{\gamma,\nu}}
g_{\alpha\beta}^{\gamma}\mathbf{1}_{\gamma}\ast f\\
&=&v^{\langle{\alpha,\beta}\rangle}\sum_{\gamma\in\mathcal{P}}
g_{\alpha\beta}^{\gamma}\mathcal{E}^{-}_{\hat{Q},\gamma}(f)
\end{eqnarray*}
and we have the desired result.
\end{proof}

\begin{proposition}[Relation (6) in Theorem \ref{MT0}]\label{lemma_6}
For any $\alpha\in\mathcal{P}$ and $i\in I$,
we have
$$\mathcal{E}^{-}_{\hat{Q},\alpha}\mathcal{E}^{+}_{\hat{Q},i}-\mathcal{E}^{+}_{\hat{Q},i}\mathcal{E}^{-}_{\hat{Q},\alpha}=
  \frac{|V_i|}{a_{\alpha}}\sum_{\beta\in\mathcal{P}}a_{\beta}\mathcal{E}^{-}_{\hat{Q},\beta}
  (v^{-\langle{\beta,i}\rangle}g_{i\beta}^{\alpha}\mathcal{K}_{\hat{Q},i}-v^{\langle{\beta,i}\rangle}g_{\beta i}^{\alpha}\mathcal{K}^{-1}_{\hat{Q},i}).$$
\end{proposition}

For the proof of Proposition \ref{lemma_6}, we need the following lemmas.

\begin{lemma}\label{lemma_6_1}
For any $i\in I$ such that $i$ is a source and $\alpha\in\mathcal{P}$ such that $V_i$ is not a direct summand of $V_{\alpha}$,
we have
$${^i\mathcal{E}}^{-}_{\hat{Q},\alpha}{^i\mathcal{E}}^{+}_{\hat{Q},i}-{^i\mathcal{E}}^{+}_{\hat{Q},i}{^i\mathcal{E}}^{-}_{\hat{Q},\alpha}=
  \frac{|V_i|}{a_{\alpha}}\sum_{\beta\in\mathcal{P}}a_{\beta}{^i\mathcal{E}}^{-}_{\hat{Q},\beta}
  (v^{-\langle{\beta,i}\rangle}g_{i\beta}^{\alpha}{^i\mathcal{K}}_{\hat{Q},i}-v^{\langle{\beta,i}\rangle}g_{\beta i}^{\alpha}{^i\mathcal{K}}^{-1}_{\hat{Q},i}).$$
\end{lemma}

\begin{proof}
Since $i$ is a source and $\alpha\in\mathcal{P}$ such that $V_i$ is not a direct summand of $V_{\alpha}$, we have
$$-\frac{|V_i|}{a_{\alpha}}\sum_{\beta\in\mathcal{P}}a_{\beta}{^i\mathcal{E}}^{-}_{\hat{Q},\beta}
  v^{\langle{\beta,i}\rangle}g_{\beta i}^{\alpha}{^i\mathcal{K}}^{-1}_{\hat{Q},i}=0.$$
Hence, it is enough to prove
\begin{equation}\label{f_3}
{^i\mathcal{E}}^{-}_{\hat{Q},\alpha}{^i\mathcal{E}}^{+}_{\hat{Q},i}-{^i\mathcal{E}}^{+}_{\hat{Q},i}{^i\mathcal{E}}^{-}_{\hat{Q},\alpha}=
  \frac{|V_i|}{a_{\alpha}}\sum_{\beta\in\mathcal{P}}a_{\beta}{^i\mathcal{E}}^{-}_{\hat{Q},\beta}
  v^{-\langle{\beta,i}\rangle}g_{i\beta}^{\alpha}{^i\mathcal{K}}_{\hat{Q},i}.
\end{equation}

Fix $\nu\in\mathbb{N}\hat{I}$ and let $\mu=\nu+\alpha$, $\mu'=\nu-i$ and $\nu'=\nu+\alpha-i$.
Fix $\hat{I}$-graded vector spaces $\mathbf{V}^{\nu}$, $\mathbf{V}^{\mu}$, $\mathbf{V}^{\nu'}$ and $\mathbf{V}^{\mu'}$ with $\underline{\dim}\mathbf{V}^{\nu}=\nu$, $\underline{\dim}\mathbf{V}^{\mu}=\mu$, $\underline{\dim}\mathbf{V}^{\nu'}=\nu'$ and $\underline{\dim}\mathbf{V}^{\mu'}=\mu'$ respectively.

Consider the following diagram
\begin{equation}\label{cd_4}
\xymatrix{
Y\ar[rr]^{\pi'}\ar[dd]_{\pi}&&{^i{E}}_{\nu',{\hat{Q}}}\\
&Y_1\ar[lu]_{r_1}\ar[r]^{\varpi_1'}\ar[d]_{\varpi_1}&{^i{Z}}_{\hat{Q}}\ar[u]_-{{^ip_2}}\ar[d]^-{{^ip_2'}}\\
{^i{E}}_{\nu,{\hat{Q}}}&{^i{Z}}_{\hat{Q},\alpha}\ar[l]_-{{^ip_1}}\ar[r]^-{{^ip_1'}}&{^i{E}}_{\mu,{\hat{Q}}}
}\end{equation}
where
\begin{enumerate}
  \item[(1)]$Y={^i{E}}_{\nu,{\hat{Q}}}\times{^i{E}}_{\nu',{\hat{Q}}}=\{(x,x')\,\,|\,\,x\in{^i{E}}_{\nu,{\hat{Q}}},x'\in{^i{E}}_{\nu',{\hat{Q}}}\}$,
  \item[(2)]$\pi(x,x')=x$ and $\pi'(x,x')=x'$,
  \item[(3)]$Y_1={^i{Z}}_{\hat{Q},\alpha}\times_{{^i{E}}_{\mu,{\hat{Q}}}}{^i{Z}}_{\hat{Q}}$, that is
      $$Y_1=\{(x,x_1,x',y,y')\,\,|\,\,(x,x_1,y)\in{^i{Z}}_{\hat{Q},\alpha},(x_1,x',y')\in{^i{Z}}_{\hat{Q}}\},$$
  \item[(4)]$\varpi_1(x,x_1,x',y,y')=(x,x_1,y)$ and $\varpi'_1(x,x_1,x',y,y')=(x_1,x',y')$,
  \item[(5)]$r_1(x,x_1,x',y,y')=(x,x')$,
  \item[(6)]${^ip_1}(x,x_1,y)=x$ and ${^ip'_1}(x,x_1,y)=x_1$,
  \item[(7)]${^ip_2}(x_1,x',y')=x'$ and ${^ip'_2}(x_1,x',y')=x_1$.
\end{enumerate}

By definition,
\begin{eqnarray*}
{^i\mathcal{E}}^{+}_{\hat{Q},i}{^i\mathcal{E}}^{-}_{\hat{Q},\alpha}&=&(-v)^{\textrm{dim}V_{\alpha}}v^{\langle{\alpha,\nu}\rangle+\textrm{dim}V_{\alpha}}v^{-\nu'_i}|G_{\mu}|^{-1}|G_{\nu}|^{-1}(^ip_2)_{!}(^ip'_2)^\ast(^ip'_1)_{!}(^ip_1)^\ast\\
&=&(-v)^{\textrm{dim}V_{\alpha}}v^{\langle{\alpha,\nu}\rangle+\textrm{dim}V_{\alpha}-\nu'_i}|G_{\mu}|^{-1}|G_{\nu}|^{-1}(^ip_2)_{!}(\varpi'_1)_{!}(\varpi_1)^\ast(^ip_1)^\ast\\
&=&(-v)^{\textrm{dim}V_{\alpha}}v^{\langle{\alpha,\nu}\rangle+\textrm{dim}V_{\alpha}-\nu'_i}|G_{\mu}|^{-1}|G_{\nu}|^{-1}(\pi')_{!}(r_1)_{!}(r_1)^\ast(\pi)^\ast.
\end{eqnarray*}

Consider the following decomposition of $Y_1=Y_{11}\bigsqcup Y_{12}$, where
$$Y_{11}=\{(x,x_1,x',y,y')\in Y_1\,\,|\,\,\textrm{Im$y\subset$Im$y'$}\}$$
and $$Y_{12}=\{(x,x_1,x',y,y')\in Y_1\,\,|\,\,\textrm{Im$y\not\subset$Im$y'$}\}.$$
Denote by $r_{11}:Y_{11}\rightarrow Y$ and $r_{12}:Y_{12}\rightarrow Y$ the restrictions of $r_1$. And we have
\begin{eqnarray}\label{f_4}
{^i\mathcal{E}}^{+}_{\hat{Q},i}{^i\mathcal{E}}^{-}_{\hat{Q},\alpha}
&=&(-v)^{\textrm{dim}V_{\alpha}}v^{\langle{\alpha,\nu}\rangle+\textrm{dim}V_{\alpha}-\nu'_i}|G_{\mu}|^{-1}|G_{\nu}|^{-1}(\pi')_{!}(r_1)_{!}(r_1)^\ast(\pi)^\ast\\
&=&(-v)^{\textrm{dim}V_{\alpha}}v^{\langle{\alpha,\nu}\rangle+\textrm{dim}V_{\alpha}-\nu'_i}|G_{\mu}|^{-1}|G_{\nu}|^{-1}(\pi')_{!}(r_{11})_{!}(r_{11})^\ast(\pi)^\ast\nonumber\\
&&+(-v)^{\textrm{dim}V_{\alpha}}v^{\langle{\alpha,\nu}\rangle+\textrm{dim}V_{\alpha}-\nu'_i}|G_{\mu}|^{-1}|G_{\nu}|^{-1}(\pi')_{!}(r_{12})_{!}(r_{12})^\ast(\pi)^\ast.\nonumber
\end{eqnarray}

Similarly, consider the following diagram
\begin{equation}\label{cd_5}
\xymatrix{
Y\ar[rr]^{\pi'}\ar[dd]_{\pi}&&{^i{E}}_{\nu',{\hat{Q}}}\\
&Y_2\ar[lu]_{r_2}\ar[r]^{\varpi'_2}\ar[d]_{\varpi_2}&{^i{Z}}_{\hat{Q},\alpha}\ar[u]_-{^ip_4'}\ar[d]^-{^ip_4}\\
{^i{E}}_{\nu,{\hat{Q}}}&{^i{Z}}_{\hat{Q}}\ar[l]_-{^ip_3'}\ar[r]^-{^ip_3}&{^i{E}}_{\mu',{\hat{Q}}}
}\end{equation}
where
\begin{enumerate}
   \item[(1)]$Y_2={^i{Z}}_{\hat{Q}}\times_{{^i{E}}_{\mu',{\hat{Q}}}}{^i{Z}}_{\hat{Q},\alpha}$, that is $$Y_2=\{(x,x_1,x',y,y')\,\,|\,\,(x,x_1,y)\in{^i{Z}}_{\hat{Q}},(x_1,x',y')\in{^i{Z}}_{\hat{Q},\alpha}\},$$
   \item[(2)]$\varpi_2(x,x_1,x',y,y')=(x,x_1,y)$ and $\varpi'_2(x,x_1,x',y,y')=(x_1,x',y')$,
   \item[(3)]$r_2(x,x_1,x',y,y')=(x,x')$,
   \item[(4)]${^ip'_3}(x,x_1,y)=x$ and ${^ip_3}(x,x_1,y)=x_1$,
   \item[(5)]${^ip'_4}(x_1,x',y')=x'$ and ${^ip_4}(x_1,x',y')=x_1$.
\end{enumerate}

By definition,
\begin{eqnarray}\label{f_5}
\mathcal{E}^{-}_{\hat{Q},\alpha}\mathcal{E}^{+}_{\hat{Q},i}&=&(-v)^{\textrm{dim}V_{\alpha}}v^{\langle{\alpha,\mu'}\rangle+\textrm{dim}V_{\alpha}}v^{-\mu'_i}|G_{\nu}|^{-1}|G_{\mu'}|^{-1}(^ip'_4)_{!}(^ip_4)^\ast(^ip_3)_{!}(^ip'_3)^\ast\\
&=&(-v)^{\textrm{dim}V_{\alpha}}v^{\langle{\alpha,\mu'}\rangle+\textrm{dim}V_{\alpha}-\mu'_i}|G_{\nu}|^{-1}|G_{\mu'}|^{-1}(^ip'_4)_{!}(\varpi'_2)_{!}(\varpi_2)^\ast(^ip'_3)^\ast\nonumber\\
&=&(-v)^{\textrm{dim}V_{\alpha}}v^{\langle{\alpha,\mu'}\rangle+\textrm{dim}V_{\alpha}-\mu'_i}|G_{\nu}|^{-1}|G_{\mu'}|^{-1}(\pi')_{!}(r_2)_{!}(r_2)^\ast(\pi)^\ast.\nonumber
\end{eqnarray}

For any $(x,x_1,x',y,y')\in Y_2$, we have the following diagram
$$
\xymatrix{
V^{\nu}_i\ar[r]^{x^i}&\hat{V}^{\nu}_i\\
V^{\mu'}_i\ar[r]^{x_1^i}\ar[u]_{y_i}\ar[d]^{y'_i}&\hat{V}^{\mu'}_i\ar[u]_{y^i}\ar[d]^{{y'}^i}\\
V^{\nu'}_i\ar[r]^{{x'}^i}&\hat{V}^{\nu'}_i\\
}$$
where $\hat{V}_i=\oplus_{h\in\hat{H},s(h)=i}V_{t(h)}$, $y^i=\oplus_{h\in\hat{H},s(h)=i}y_{t(h)}$ and $x^i=\oplus_{h\in\hat{H},s(h)=i}x_{h}$.
Note that $y^i$ is an isomorphism.

Let $Y_{21}$ be the subset of $Y_2$ such that $$\textrm{Im${y'}^i(y^i)^{-1}x^i\subset$Im${x'}^i$}$$ and $Y_{22}=Y_2\backslash Y_{21}$. This condition implies that there is an embedding $\iota$ from  $\mathbf{V}^{\nu}$ to $\mathbf{V}^{\nu'}$ such that the following diagram is commutative
for any $(x,x_1,x',y,y')\in Y_{21}$
$$
\xymatrix{
V^{\mu'}_i\ar[r]^{x_1^i}\ar@(dl,ul)[dd]_{y'_i}\ar[d]^{y_i}&\hat{V}^{\mu'}_i\ar@(dr,ur)[dd]^{{y'}^i}\ar[d]^{{y}^i}\\
V^{\nu}_i\ar[r]^{x^i}\ar[d]^{{\iota}_i}&\hat{V}^{\nu}_i\ar[d]^{{\iota}^i}\\
V^{\nu'}_i\ar[r]^{{x'}^i}&\hat{V}^{\nu'}_i
.}$$
Since $i$ is a source and $\alpha\in\mathcal{P}$ such that $V_i$ is not a direct summand of $V_{\alpha}$, the set
$Y_{21}$ is empty and $Y_{2}=Y_{22}$.

Since $G_{\mu}$ (resp. $G_{\mu'}$) acts on $Y_{12}$ (resp. $Y_{22}$) freely and there is a bijection between $Y_{12}/G_{\mu}$ and $Y_{22}/G_{\mu'}$ by Lemma \ref{lemma_10}, we have $$|G_{\mu}|^{-1}(\pi')_{!}(r_{12})_{!}(r_{12})^\ast(\pi)^\ast=|G_{\mu'}|^{-1}(\pi')_{!}(r_{2})_{!}(r_{2})^\ast(\pi)^\ast.$$
Note that
$$v^{\langle{\alpha,\nu}\rangle+\textrm{dim}V_{\alpha}-\nu'_i}=v^{\langle{\alpha,\mu'}\rangle+\textrm{dim}V_{\alpha}-\mu'_i}.$$
We have
\begin{eqnarray}\label{f_6}
&&v^{\langle{\alpha,\nu}\rangle+\textrm{dim}V_{\alpha}-\nu'_i}|G_{\mu}|^{-1}|G_{\nu}|^{-1}(\pi')_{!}(r_{12})_{!}(r_{12})^\ast(\pi)^\ast\\
&=&v^{\langle{\alpha,\mu'}\rangle+\textrm{dim}V_{\alpha}-\mu'_i}|G_{\nu}|^{-1}|G_{\mu'}|^{-1}(\pi')_{!}(r_{2})_{!}(r_{2})^\ast(\pi)^\ast.\nonumber
\end{eqnarray}

Formula (\ref{f_4}),  (\ref{f_5}) and (\ref{f_6}) imply that
$${^i\mathcal{E}}^{-}_{\hat{Q},\alpha}{^i\mathcal{E}}^{+}_{\hat{Q},i}-{^i\mathcal{E}}^{+}_{\hat{Q},i}{^i\mathcal{E}}^{-}_{\hat{Q},\alpha}=-(-v)^{\textrm{dim}V_{\alpha}}v^{\langle{\alpha,\nu}\rangle+\textrm{dim}V_{\alpha}-\nu'_i}|G_{\mu}|^{-1}|G_{\nu}|^{-1}(\pi')_{!}(r_{11})_{!}(r_{11})^\ast(\pi)^\ast.$$
Hence, for the proof of Formula (\ref{f_3}), it is enough to prove
\begin{eqnarray}\label{f_7}
&&(-v)^{\textrm{dim}V_{\alpha}}v^{\langle{\alpha,\nu}\rangle+\textrm{dim}V_{\alpha}-\nu'_i}|G_{\mu}|^{-1}|G_{\nu}|^{-1}(\pi')_{!}(r_{11})_{!}(r_{11})^\ast(\pi)^\ast\\
&=&-\frac{|V_i|}{a_{\alpha}}\sum_{\beta\in\mathcal{P}}a_{\beta}\mathcal{E}^{-}_{\hat{Q},\beta}
  v^{-\langle{\beta,i}\rangle}g_{i\beta}^{\alpha}{^i\mathcal{K}}_{\hat{Q},i}.\nonumber
\end{eqnarray}

By the commutative diagram (\ref{cd_4}), we have
$$(\pi')_{!}(r_{11})_{!}(r_{11})^\ast(\pi)^\ast=(^ip_2)_{!}(\varpi'_{11})_{!}(\varpi_{11})^\ast(^ip_1)^\ast
,$$
where $\varpi'_{11}$ and $\varpi_{11}$ are the restrictions of $\varpi'_{1}$ and $\varpi_{1}$ to $Y_{11}$ respectively.

Consider the following decomposition of $Y_{11}=\bigsqcup_{\beta\in\mathcal{P}}Y_{11\beta}$, where
$$Y_{11\beta}=\{(x,x_1,x',y,y')\in Y_{11}\,\,|\,\,[\textrm{Im}y'/\textrm{Im}y]=\beta\}.$$
Denote by $\varpi'_{11\beta}$ and $\varpi_{11\beta}$ the restrictions of $\varpi'_{11}$ and $\varpi_{11}$ to $Y_{11\beta}$ respectively
and we have
$$(^ip_2)_{!}(\varpi'_{11})_{!}(\varpi_{11})^\ast(^ip_1)^\ast
=\sum_{\beta\in\mathcal{P}}(^ip_2)_{!}(\varpi'_{11\beta})_{!}(\varpi_{11\beta})^\ast(^ip_1)^\ast.$$

Thus, for the proof of Formula (\ref{f_7}), we just need to prove
\begin{eqnarray}\label{f_8}
&&(-v)^{\textrm{dim}V_{\alpha}}v^{\langle{\alpha,\nu}\rangle+\textrm{dim}V_{\alpha}-\nu'_i}|G_{\mu}|^{-1}|G_{\nu}|^{-1}(^ip_2)_{!}(\varpi'_{11\beta})_{!}(\varpi_{11\beta})^\ast(^ip_1)^\ast\\
&=&-\frac{|V_i|a_{\beta}}{a_{\alpha}}\mathcal{E}^{-}_{\hat{Q},\beta}
  v^{-\langle{\beta,i}\rangle}g_{i\beta}^{\alpha}{^i\mathcal{K}}_{\hat{Q},i}.\nonumber
  \end{eqnarray}

Denote by $\tau$ the canonical projection from $Y_{11\beta}$ to ${^iZ}_{\hat{Q},\beta}$.
By Lemma \ref{lemma_8}, the cardinality of the fiber of $\tau$ at $a\in{^iZ}_{\hat{Q},\beta}$ is $|G_{\mu}|g_{i\beta}^{\alpha}\frac{a_{\beta}}{a_{\alpha}}v^{-2\langle{i,\nu}\rangle}$.
At the same time, $$-|V_i|(-v)^{\textrm{dim}V_{\beta}}v^{\langle{\beta,\nu}\rangle+\textrm{dim}V_{\beta}}
  v^{-\langle{\beta,i}\rangle}v^{\bar{\nu}_i-\nu_i}v^{2\langle{i,\nu}\rangle}=(-v)^{\textrm{dim}V_{\alpha}}v^{\langle{\alpha,\nu}\rangle+\textrm{dim}V_{\alpha}-\nu'_i}.$$
Hence
\begin{eqnarray*}
&&-\frac{|V_i|a_{\beta}}{a_{\alpha}}\mathcal{E}^{-}_{\hat{Q},\beta}
  v^{-\langle{\beta,i}\rangle}g_{i\beta}^{\alpha}{^i\mathcal{K}}_{\hat{Q},i}\\
&=&-|V_i|(-v)^{\textrm{dim}V_{\beta}}v^{\langle{\beta,\nu}\rangle+\textrm{dim}V_{\beta}}|G_{\nu}|^{-1}
  v^{-\langle{\beta,i}\rangle}v^{\bar{\nu}_i-\nu_i}|G_{\mu}|^{-1}v^{2\langle{i,\nu}\rangle}\cdot\\
&&(^ip_2)_{!}(\varpi'_{11\beta})_{!}(\varpi_{11\beta})^\ast(^ip_1)^\ast\\
&=&(-v)^{\textrm{dim}V_{\alpha}}v^{\langle{\alpha,\nu}\rangle+\textrm{dim}V_{\alpha}-\nu'_i}|G_{\mu}|^{-1}|G_{\nu}|^{-1}(^ip_2)_{!}(\varpi'_{11\beta})_{!}(\varpi_{11\beta})^\ast(^ip_1)^\ast
\end{eqnarray*}
and we prove Formula (\ref{f_8}).
\end{proof}

\begin{lemma}\label{lemma_10}
With the notations in the proof of Lemma \ref{lemma_6_1}, there is a bijection between $Y_{12}/G_{\mu}$ and $Y_{22}/G_{\mu'}$.
\end{lemma}
\begin{proof}
For any $(x,x_1,x',y,y')\in Y_{12}$, consider the following pull-back
$$
\xymatrix{
(\hat{x}_1,\mathbf{W})\ar[r]^{\hat{y}'}\ar[d]_{\hat{y}}&(x',\mathbf{V}^{\nu'})\ar[d]_{y'}\\
(x,\mathbf{V}^{\nu})\ar[r]^{y}&(x_1,\mathbf{V}^{\mu}).\\
}$$
By the definition of $Y_{12}$, $\underline{\dim}\mathbf{W}=\mu'$.
Fix a linear isomorphism $\sigma:\mathbf{V}^{\mu'}\rightarrow\mathbf{W}$. By the definition of $Y_{22}$, the element $(x,\sigma^{-1}\hat{x}_1\sigma,x',y\sigma,y'\sigma)\in Y_{22}$.
This induces a map $\theta$ from $Y_{12}$ to $Y_{22}/G_{\mu'}$.
For two elements $(x,x_1,x',y,y')$ and $(x,x_1,\sigma x'\sigma^{-1},\sigma y,\sigma y')\in Y_{12}$ with $\sigma\in G_{\mu}$,
$\theta(x,x_1,x',y,y')=\theta(x,x_1,\sigma x'\sigma^{-1},\sigma y,\sigma y')$. Hence, the map $\theta$ induces a map $\bar{\theta}$ from $Y_{12}/G_{\mu}$ to $Y_{22}/G_{\mu'}$.

For any $(x,x_1,x',y,y')\in Y_{22}$, consider the following push-out
$$
\xymatrix{
({x}_1,\mathbf{V}^{\mu'})\ar[r]^{{y}'}\ar[d]_{{y}}&(x',\mathbf{V}^{\nu'})\ar[d]_{\hat{y}'}\\
(x,\mathbf{V}^{\nu})\ar[r]^{\hat{y}}&(\hat{x}_1,\mathbf{U}).\\
}$$
By the definition of $Y_{22}$, $\underline{\dim}\mathbf{U}=\mu$.
Similarly to $\bar{\theta}$, we have a map $\bar{\vartheta}$ from $Y_{22}/G_{\mu'}$ to $Y_{12}/G_{\mu}$.

By the relation between pull-back and push-out, we have $\bar{\vartheta}$ is the inverse of $\bar{\theta}$ and $\bar{\theta}$ is a bijection.
\end{proof}

\begin{lemma}\label{lemma_8}
With the notations in the proof of Lemma \ref{lemma_6_1}, the cardinality of any fiber of $\tau$ is $v^{-2\langle{i,\nu}\rangle}|G_{\mu}|g_{i\beta}^{\alpha}\frac{a_{\beta}}{a_{\alpha}}.$
\end{lemma}
\begin{proof}
Consider the following maps (\cite{green1995hall,Schiffmann_Lectures1,XXZ_Ringel-Hall_algebras_beyond_their_quantum_groups})
$$
\xymatrix{
B&C\ar[l]_-{\pi_1}\ar[r]^-{\pi_2}&D\ar[r]^-{\pi_3}&{^iZ}_{\hat{Q}}
}.$$
Here
\begin{enumerate}
   \item[(1)]$B$ is set of
      $$
      \xymatrix{
      &&0\ar[d]\\
      &&L\ar[d]^{a'}\\
      0\ar[r]&N\ar[r]^{a}&J
      }$$
   where $L,N,J$ are representations of the quiver $\hat{Q}$, $a,a'$ are homomorphisms of representations and the row and column are exact,
   \item[(2)]$C$ is set of
      $$
      \xymatrix{
      &&0\ar[d]&&\\
      &&L\ar[d]^{a'}&&\\
      0\ar[r]&N\ar[r]^{a}&J\ar[d]^{b'}\ar[r]^{b}&M\ar[r]&0\\
      &&K\ar[d]&&\\
      &&0&&
      }$$
   where $L,N,J,M,K$ are representations of the quiver $\hat{Q}$, $a,a'$, $b,b'$ are homomorphisms of representations and the row and column are exact,
   \item[(3)]$D$ is set of
      $$
      \xymatrix{
      &0\ar[d]&&0\ar[d]&\\
      0\ar[r]&L_1\ar[d]^{u'}\ar[r]^{u}&L\ar[r]^{v}&L_2\ar[d]^{x}\ar[r]&0\\
      &N\ar[d]^{v'}&&M\ar[d]^{y}&\\
      0\ar[r]&K_1\ar[r]^{x'}\ar[d]&K\ar[r]^{y'}&K_2\ar[d]\ar[r]&0\\
      &0&&0&
      }$$
   where $L,N,M,K,L_1,L_2,K_1,K_2$ are representations of the quiver $\hat{Q}$, $u,u'$, $v,v'$, $x,x'$, $y,y'$ are homomorphisms of representations and the rows and columns are exact,
   \item[(4)]$\pi_1$ is the canonical projection,
   \item[(5)]$\pi_2(J,a,a',b,b')=(L_1,L_2,K_1,K_2,u,u',v,v',x,x',y,y')$, where $L_1=\textrm{Ker}b'a=\textrm{Ker}ba'$, $L_2=\textrm{Im}ba'$, $K_1=\textrm{Im}b'a$, $K_2=\textrm{Coker}b'a=\textrm{Coker}ba'$, $u,u'$ and $x,x'$ are natural embeddings, $v,v'$ and $x,x'$ are natural projections,
   \item[(6)]$\pi_1(L_1,L_2,K_1,K_2,u,u',v,v',x,x',y,y')=(L_1,L,u)$.
\end{enumerate}

Fix $N,L,M,K$ with $\underline{\dim}N=\nu,\underline{\dim}L=\nu+\beta,\underline{\dim}K=i$ and $[M]=\alpha$.
Then for any $x\in{^iZ}_{\hat{Q},\beta}$, $$|\pi_{3}^{-1}(x)|=|G_{\beta}|g_{i\beta}^{\alpha}a_{\beta}a_{i}|G_{\nu}||G_{i}|.$$
Similarly to proof of Green formula (\cite{green1995hall,Schiffmann_Lectures1,XXZ_Ringel-Hall_algebras_beyond_their_quantum_groups}), $$|\pi_{2}^{-1}\pi_{3}^{-1}(x)|=g_{i\beta}^{\alpha}a_{\beta}a_{i}|G_{\mu}|\frac{|\textrm{Ext}(K_2,L_1)|}{|\textrm{Hom}(K_2,L_1)|}
=g_{i\beta}^{\alpha}a_{\beta}a_{i}|G_{\mu}|v^{-2\langle{i,\nu}\rangle}.$$
Hence $$|\pi_{1}\pi_{2}^{-1}\pi_{3}^{-1}(x)|
=\frac{a_{\beta}}{a_{\alpha}}g_{i\beta}^{\alpha}|G_{\mu}|v^{-2\langle{i,\nu}\rangle}.$$

Since $\pi_{1}\pi_{2}^{-1}\pi_{3}^{-1}(x)=\tau^{-1}(x)$,
the cardinality of the fiber $\tau$ at $x$ is $\frac{a_{\beta}}{a_{\alpha}}g_{i\beta}^{\alpha}|G_{\mu}|v^{-2\langle{i,\nu}\rangle}.$
\end{proof}

\begin{lemma}[Theorem 3.9 (6) in \cite{zheng2008categorification}]\label{lemma_9}
For any $i\in I$,
we have
$${\mathcal{E}}^{-}_{\hat{Q},i}{\mathcal{E}}^{+}_{\hat{Q},i}-{\mathcal{E}}^{+}_{\hat{Q},i}{\mathcal{E}}^{-}_{\hat{Q},i}=
  \frac{v^2}{v^2-1}
  ({\mathcal{K}}_{\hat{Q},i}-{\mathcal{K}}^{-1}_{\hat{Q},i})$$
\end{lemma}

Now, we can give the proof of Proposition \ref{lemma_6}.
\begin{proof}[Proof of Proposition \ref{lemma_6}]
By the definition of ${\mathcal{E}}^{+}_{\hat{Q},i}$, it is enough to prove this when $i$ is a source.

At first, assume that  $\alpha\in\mathcal{P}$ such that $V_i$ is not a direct summand of $V_{\alpha}$.

Lemma \ref{lemma_6_1} and commutative diagrams (\ref{cd_1})-(\ref{cd_3}) imply that
\begin{equation}\label{formula_1}
({\mathcal{E}}^{-}_{\hat{Q},\alpha}{\mathcal{E}}^{+}_{\hat{Q},i}-{\mathcal{E}}^{+}_{\hat{Q},i}{\mathcal{E}}^{-}_{\hat{Q},\alpha})(f)=
  \frac{|V_i|}{a_{\alpha}}\sum_{\beta\in\mathcal{P}}a_{\beta}{\mathcal{E}}^{-}_{\hat{Q},\beta}
  (v^{-\langle{\beta,i}\rangle}g_{i\beta}^{\alpha}{\mathcal{K}}_{\hat{Q},i}-v^{\langle{\beta,i}\rangle}g_{\beta i}^{\alpha}{\mathcal{K}}^{-1}_{\hat{Q},i})(f)
\end{equation}
for any $f\in\mathcal{F}_{G_{\nu,I}}(E_{\nu,\hat{Q}})$ such that $\mathrm{supp}f\subset{^iE}_{\nu,\hat{Q}}$,

Formula (\ref{formula_1}) and Proposition \ref{proposition_zheng_K_2}, \ref{proposition_zheng_e_1}, \ref{proposition_N_f} imply that
\begin{equation}\label{formula_2}
({\mathcal{E}}^{-}_{\hat{Q},\alpha}{\mathcal{E}}^{+}_{\hat{Q},i}-{\mathcal{E}}^{+}_{\hat{Q},i}{\mathcal{E}}^{-}_{\hat{Q},\alpha})([f])=
  \frac{|V_i|}{a_{\alpha}}\sum_{\beta\in\mathcal{P}}a_{\beta}{\mathcal{E}}^{-}_{\hat{Q},\beta}
  (v^{-\langle{\beta,i}\rangle}g_{i\beta}^{\alpha}{\mathcal{K}}_{\hat{Q},i}-v^{\langle{\beta,i}\rangle}g_{\beta i}^{\alpha}{\mathcal{K}}^{-1}_{\hat{Q},i})([f])
  \end{equation}
for any $[f]\in\mathcal{F}_{\nu,\hat{Q}}$.

In general, for any $\alpha\in\mathcal{P}$, $V_{\alpha}\simeq kV_i\oplus V_{\alpha'}$ such that $V_i$ is not a direct summand of $V_{\alpha'}$, we have $${\mathcal{E}}^{-}_{\hat{Q},\alpha}=v^{-\langle\alpha',ki\rangle}{\mathcal{E}}^{-}_{\hat{Q},\alpha'}{\mathcal{E}}^{-}_{\hat{Q},ki}.$$
Hence, Formula (\ref{formula_2}) and Lemma \ref{lemma_9} imply the desired result.
\end{proof}

\bibliography{mybibfile}

\end{document}